\documentclass[12pt]{amsart}
\usepackage{colordvi,amssymb,amsmath,graphicx}

\numberwithin{equation}{section}
\newtheorem{theorem}[equation]{Theorem}
\newtheorem{lemma}[equation]{Lemma}

\newtheorem{proposition}[equation]{Proposition}

\newtheorem{reduction}[equation]{Reduction}
\newcounter{FNC}[page]
\def\fauxfootnote#1{{\addtocounter{FNC}{2}$^\fnsymbol{FNC}$%
     \let\thefootnote\relax\footnotetext{$^\fnsymbol{FNC}$\Magenta{#1}}}}
\newcommand{\simplex}{\includegraphics{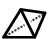}}
\newcommand{\ssimplex}{\includegraphics{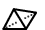}}
\newcommand{\dcol}[1]{\Blue{#1}}
\newcommand{\demph}[1]{\dcol{{\sl #1}}}
\newcommand{\conv}{\operatorname{conv}}
\newcommand{\ubc}{\operatorname{ubc}}
\newcommand{\vol}{\operatorname{vol}}
\newcommand{\C}{\mathbb{C}}
\newcommand{\G}{\mathbb{G}}

\renewcommand{\P}{\mathbb{P}}
\newcommand{\R}{\mathbb{R}}
\newcommand{\T}{\mathbb{T}}
\newcommand{\Z}{\mathbb{Z}}
\newcommand{\calA}{{\mathcal{A}}}
\newcommand{\calB}{{\mathcal{B}}}

\newcommand{\calM}{{\mathcal{M}}}

\newcommand{\calV}{{\mathcal{V}}}
\newcommand{\calW}{{\mathcal{W}}}

\title{Dense Fewnomials}

\author[K.~Rusek]{Korben Rusek}
\address{Department of Mathematics\\
         Texas A\&M University\\
         College Station\\
         TX \ 77843}
\email{krusek@math.tamu.edu}
\urladdr{http://www.math.tamu.edu/\~{}krusek}
\author[J.~Shakalli]{Jeanette Shakalli}
\address{Department of Mathematics\\
         Texas A\&M University\\
         College Station\\
         TX \ 77843}
\email{shakalli@math.tamu.edu}
\urladdr{http://www.math.tamu.edu/\~{}shakalli}
\author[F.~Sottile]{Frank Sottile}
\address{Department of Mathematics\\
         Texas A\&M University\\
         College Station\\
         TX \ 77843}
\email{sottile@math.tamu.edu}
\urladdr{http://www.math.tamu.edu/\~{}sottile}
\thanks{Research of Sottile supported in part by NSF grants DMS-070105 and
  DMS-1001615.}
\thanks{Research of Rusek supported in part by DOE ASCR grant DE-SC0002505}
\begin{document}
%%%%%%%%%%%%%%%%%%%%%%%%%%%%%%%%%%%%%%%%%%%%%%%%%%%%%%%%%%%%%%%

\begin{abstract}
  We derive new bounds of fewnomial type for the number of real solutions to
  systems of polynomials that have structure intermediate between fewnomials and 
  generic (dense) polynomials.
  This uses a modified version of Gale duality for polynomial systems.
  We also use stratified Morse theory to bound the total Betti number of a hypersurface
  defined by such a dense fewnomial.
  These bounds contain and generalize previous bounds for
  ordinary fewnomials obtained by Bates, Bertrand, Bihan, and Sottile.
\end{abstract}
%%%%%%%%%%%%%%%%%%%%%%%%%%%%%%%%%%%%%%%%%%%%%%%%%%%%%%%%%%%%%%%
\maketitle
%%%%%%%%%%%%%%%%%%%%%%%%%%%%%%%%%%%%%%%%%%%%%%%%%%%%%%%%%%%%%%%%%%%%%%%%%%%%%

%%%%%%%%%%%%%%%%%%%%%%%%%%%%%%%%%%%%%%%%%%%%%%%%%%%%%%%%%%%%%%%%%%%%%%%%%%%%%
%
\section{Introduction}

The classical theorem of B\'ezout~\cite{Bezout} bounds the number of solutions
to a system of polynomials by the product of their degrees.
While this B\'ezout bound is sharp for generic systems of polynomial equations,
that is no longer the case when the equations possess additional structure.
For example, Kushnirenko~\cite{BKK} showed that if all the polynomials have
the same Newton polytope, then the number of nondegenerate solutions to
such a system is at most the volume of the Newton polytope, suitably normalized.  

Bounds for the number of nondegenerate real solutions are governed by
Kushnirenko's ``fewnomial principle'': roughly, few monomials implies few
solutions or restricted topology~\cite{Kush_note}. 
This principle was established by Khovanskii in his fundamental work on
fewnomials~\cite{Kh91} in which he showed that a system of $n$ polynomials in $n$ variables
where the polynomials have $1{+}k{+}n$ distinct monomials has fewer than
 \begin{equation}\label{Eq:Khovanskii_bound}
   2^{\binom{k{+}n}{2}}\, (n{+}1)^{k{+}n}
 \end{equation}
nondegenerate positive solutions.
This bound is remarkable as it is independent of the degrees and Newton polytopes of the
polynomials, which control the number of complex solutions.
Few of the complex solutions to a fewnomial system can be real.

Khovanskii's bound was lowered by Bihan and Sottile~\cite{BS_Khovanskii} to
 \begin{equation}\label{Eq:BS}
   \frac{e^2{+}3}{4}\, 2^{\binom{k}{2}}\, n^k\,.
 \end{equation}
For this, they transformed the original polynomial system into an
equivalent Gale dual~\cite{BS_Gale} system of rational functions, whose number of solutions 
they bounded.
An essential step for the bound on Gale systems uses Khovanskii's generalization of the
classical Rolle Theorem.

The bound~\eqref{Eq:BS} is smaller than Khovanskii's
bound~\eqref{Eq:Khovanskii_bound} because Khovanskii's bound is a specialization to
polynomials of a bound for more general functions and the proof
of~\eqref{Eq:BS} takes advantage of some special geometry enjoyed by polynomials. 

We derive bounds of fewnomial type for polynomial systems with structure 
intermediate between that of fewnomials and general polynomials.
These bounds can be dramatically smaller than the fewnomial bound~\eqref{Eq:BS}, and like 
that bound do not depend upon the degrees of the polynomials involved.
A collection $\calA\subset\Z^n$ of exponent vectors is 
\demph{$(d,\ell)$-dense} if there are integers $d,\ell$ such that 
$\calA$ admits a decomposition of the form
  \begin{equation}\label{Eq:strFew}
   \calA\ =\ \psi(d\simplex^\ell\cap\Z^\ell)\; \bigcup\; \calW\,,
  \end{equation}
where $\calW$ consists of $n$ affinely independent vectors, $\psi\colon\Z^\ell\to\Z^n$
is an affine-linear map, and $\simplex^\ell$ is the unit simplex in $\R^\ell$.

A Laurent polynomial whose support $\calA$ is $(d,\ell)$-dense~\eqref{Eq:strFew} is a  
\demph{$(d,\ell)$-dense fewnomial}.
We show that a system of $(d,\ell)$-dense fewnomials
in $n$ variables has fewer than
 \begin{equation}\label{Eq:structured_bound}
   \frac{e^2{+}3}{4}\, 2^{\binom{\ell }{2}}\, n^\ell \cdot d^\ell
 \end{equation}
nondegenerate positive solutions.
To compare this to~\eqref{Eq:BS}, observe that $k{+}1\ =\ \binom{d+\ell}{\ell}$,
so the bounds coincide when $d=1$, but otherwise~\eqref{Eq:structured_bound} is dramatically
smaller.

In~\cite{BBS_all}, the methods of~\cite{BS_Khovanskii} were extended to establish the bound
 \begin{equation}\label{Eq:BBS_bound}
   \frac{e^{\Red{4}}+3}{4}\, 2^{\binom{k}{2}}\, n^k\,,
 \end{equation}
for the number of nonzero real solutions to a fewnomial system---not just positive
solutions---when the exponent vectors span a sublattice of $\Z^n$ of odd index.
The same arguments show that if the exponent vectors $\calA$ span a
sublattice of $\Z^n$ of odd index, then the number of nondegenerate real solutions to a system
of $(d,\ell)$-dense fewnomials is at most
 \begin{equation}\label{Eq:structured_real}
   \frac{e^{\Red{4}}+3}{4}\, 2^{\binom{\ell }{2}}\, n^\ell \cdot d^\ell\,.
 \end{equation}

Khovanskii also gave a bound for the sum \demph{$b_*(X)$} of the Betti 
numbers of a smooth hypersurface in the positive orthant $\R^n_>$ defined by a polynomial
with $1{+}k{+}n$ monomial terms~\cite{Kh91} (Corollary 4, p. 91),
 \begin{equation}\label{eq:Khovanskii}
    b_*(X)\ \leq\   (2n^2{-}n{+}1)^{k{+}n}\, (2n)^{n{-}1}\, 2^{\binom{k{+}n}{2}}\,.
 \end{equation}
Bihan and Sottile~\cite{BS_Betti} used the fewnomial bound and stratified Morse theory for
a manifold with corners~\cite{GM88} to lower this to
 \begin{equation}\label{eq:Betti}
    b_*(X) \ <\   \frac{e^2{+}3}{4}\, 2^{\binom{k}{2}}
    \cdot \sum_{i=0}^{n} \tbinom{n}{i}\, i^k\,.
 \end{equation}
The same arguments show that when $X$ is defined by a $(d,\ell)$-dense fewnomial, we have 
 \begin{equation}\label{Eq:structured_Betti}
  b_*(X)\ <\ \frac{e^2{+}3}{4}\, 2^{\binom{\ell}{2}} \cdot d^\ell
    \cdot \sum_{i=0}^{n} \tbinom{n}{i}\, i^\ell\,.
 \end{equation}

An important step in these arguments is a version of Gale duality for
dense fewnomials, which generalizes Gale duality for polynomial
systems as established in~\cite{BS_Gale}.

These bounds~\eqref{Eq:structured_bound},~\eqref{Eq:structured_real},
and~\eqref{Eq:structured_Betti} simultaneously generalize the results
of~\cite{BS_Khovanskii},~\cite{BBS_all}, 
and~\cite{BS_Betti}, which are the cases when $d=1$.
The case of $\ell=1$ of the bound~\eqref{Eq:BBS_bound} was established
in~\cite{BBS_few}, where a $(d,1)$-dense fewnomial was called a near circuit.

This paper is structured as follows.
We begin in Section~\ref{Sec:one} with definitions and examples of dense fewnomials, give the
precise statements of our main theorems, and study an example when $n=2$.
Section~\ref{S:gale} is devoted to establishing the variant of Gale duality appropriate for
dense fewnomials, and in Section~\ref{S:bounds}, which is the heart of this paper, we establish 
the bounds~\eqref{Eq:structured_bound} and~\eqref{Eq:structured_real}.
We develop the necessary tools and give the proof of our bound~\eqref{Eq:structured_Betti}
for the sum of Betti numbers in Section~\ref{Sec:Betti}.

We thank Maurice Rojas, who suggested looking for extentions of the fewnomial bound as a
class project to Rusek and Shakalli.

%%%%%%%%%%%%%%%%%%%%%%%%%%%%%%%%%%%%%%%%%%%%%%%%%%%%%%%%%%%%%%%%%%%%%%%%%%%%
%
\section{Dense Fewnomials}\label{Sec:one}

An integer vector $\alpha=(a_1,\dotsc,a_n)\in\Z^n$ is the exponent of a Laurent
monomial,
\[
  x^\alpha\ :=\ x_1^{a_1} x_2^{a_2}\dotsb x_n^{a_n}\,.
\]
A polynomial $f$ with \demph{support} $\calA\subset\Z^n$ is one whose exponent vectors
lie in $\calA$, 
\[
  f\ =\ \sum_{\alpha\in\calA} a_\alpha x^\alpha
  \qquad(a_\alpha\in\R)\,.
\]
We are interested in systems of $n$ polynomials, each with support $\calA$.
We obtain novel bounds on the number of real solutions to a system of polynomials when the set
of exponent vectors has structure that is intermediate between fewnomials
and dense polynomials.

Kushnirenko~\cite{BKK} showed that a general system of polynomials, all with support $\calA$,
will have $n!\vol(\conv(\calA))$ complex solutions, the normalized volume of the convex hull of 
the exponent vectors.
While this is also a bound for the number of real solutions, there is another bound which
depends only upon the number of exponents.
Specifically, a \demph{fewnomial system} is one in which the support $\calA$ consists of
$1{+}k{+}n$ monomials, but is otherwise unstructured, and such a system has the
bound~\eqref{Eq:BS} on its number of positive solutions.
When the exponents  $\calA$ affinely span a sublattice of odd index in $\Z^n$, a fewnomial
system has the bound~\eqref{Eq:BBS_bound} for its number of nondegenerate nonzero real solutions.

A \demph{dense fewnomial} is a polynomial whose support $\calA$ is intermediate
between fewnomials and general polynomials in the following way.
Let $d,\ell,n$ be positive integers, $\simplex^\ell\subset\R^\ell$ be the standard 
unit simplex, the convex hull of the original and the $\ell$ unit basis vectors, and
$\psi\colon\Z^\ell\to\Z^n$ be an affine-linear map. 
A \demph{$(d,\ell)$-dense fewnomial} is a Laurent polynomial $f$ whose support
$\calA\subset\Z^n$ admits a decomposition
 \begin{equation}\label{Eq:near_polyhedral}
  \calA\ =\ \psi(d\simplex^\ell\cap\Z^\ell)\;\bigcup\; \calW\,,
 \end{equation}
where $\calW=\{w_1,\dotsc,w_n\}$ consists of $n$ affinely independent vectors.
Such a set $\calA$~\eqref{Eq:near_polyhedral} is \demph{$(d,\ell)$-dense}. 

We give some examples of $(d,\ell)$-dense sets of exponent vectors.
 \begin{enumerate}
  \item Any collection $\calA$ of $1{+}\ell{+}n$ exponent vectors that affinely spans
    $\R^n$ is $(1,\ell)$-dense.
    To see this, let $\calW\subset\calA$ be $n$ affinely independent vectors.
    Writing $\calA-\calW=\{v_0,v_1,\dotsc,v_\ell\}$, these vectors 
    are the image of the integer points in $\simplex^\ell$ under the affine map $\psi$ that
    takes the $i$th unit vector to $v_i$ and the origin to $v_0$.

    Thus ordinary fewnomials with $1{+}k{+}n$ monomials are $(1,k)$-dense fewnomials.

  \item When $\ell=1$, the exponent vectors $\calA$ of a dense
    fewnomial form a \demph{near circuit} in the terminology of~\cite{BBS_few}.
    There, it was shown that if $\calA$ spans $\Z^n$, then a system with support $\calA$
    has at most $2dn{+}1$ nonzero real solutions.

  \item A general $(d,\ell)$-dense set $\calA$ has the following form,
 \begin{equation}\label{Eq:general_form}
   \calA\ :=\ 
    \Bigl\{ v_0\ +\ \sum_{m=1}^\ell \lambda_mv_m \mid 
            0\leq \lambda_m\,,\ \sum_i\lambda_i\leq d \Bigr\}
    \bigcup \calW\,,
 \end{equation}
    where $\calW=\{w_1,\dotsc,w_n\}\subset\Z^n$ is affinely independent and 
    $v_0,v_1,\dotsc,v_\ell$ are integer vectors.
    Below is an example of such a set $\calA$ in $\Z^2$ ($v_0=(0,0)$ is the open circle). 
 \begin{equation}\label{Eq:triangle}
  \raisebox{-38pt}{\begin{picture}(157,76)(-11,0)
   \put(0,0){\includegraphics{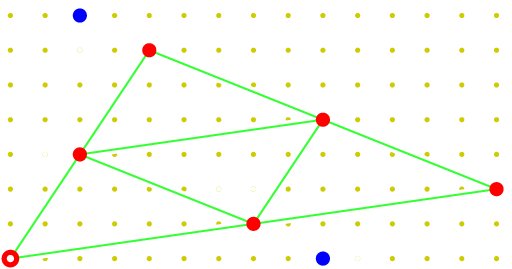}}

   \put(-11,1){$v_0$}  \put(64,21){$v_1$}
   \put(7,32){$v_2$}  
   \put(18,63){$w_2$}       \put(98,1){$w_1$}
  \end{picture}}
 \end{equation}
  For this, $n=\ell=d=2$, $\calW=\{(9,0),(2,7)\}$ and 
  $v_1=(7,1)$ and $v_2=(2,3)$.
  Here is a $(2,3)$-dense set in $\Z^2$.
\[
  \begin{picture}(187,106)(-11,0)
    \put(0,0){\includegraphics{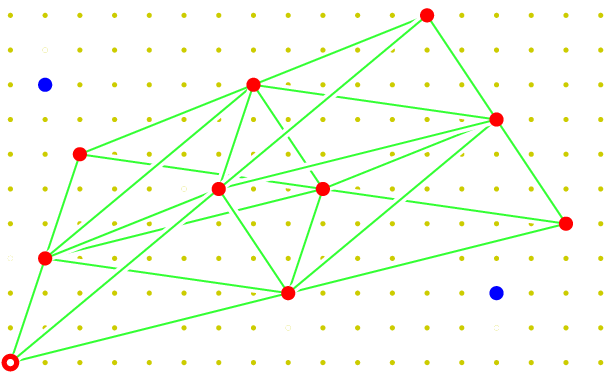}}
   \put(-3,32){$v_2$}  \put(78,13){$v_1$}
   \put(45,53){$v_3$}  \put(-11,1){$v_0$} 
   \put(8,89){$w_2$}   \put(138,13){$w_1$}
  \end{picture}
\]
 \end{enumerate}
%

%%%%%%%%%%%%%%%%%%%%%%%%%%%%%%%%%%%%%%%%%%%%%%%%%%%%%%%%%%%%%%%%%%%%%%%%%%%%
\begin{theorem}\label{Th:Structured_bound}
 Suppose that $\calA\subset\Z^n$ is $(d,\ell)$-dense.
 Then a system
\[
  f_1(x_1,\dotsc,x_n)\ =\ 
  f_2(x_1,\dotsc,x_n)\ =\ \dotsb \ =\ 
  f_n(x_1,\dotsc,x_n)\ =\ 0
\]
 of real polynomials with support $\calA$ has fewer than 
\[
   \frac{e^2{+}3}{4}\, 2^{\binom{\ell}{2}}\, n^\ell \cdot d^\ell
\]
  nondegenerate positive solutions.
  If the affine span of $\calA$ is a sublattice of\/ $\Z^n$ with odd index, then
  the number of nondegenerate real solutions is less than
\[
   \frac{e^4{+}3}{4}\, 2^{\binom{\ell}{2}}\, n^\ell \cdot d^\ell\,.
\]
\end{theorem}
%%%%%%%%%%%%%%%%%%%%%%%%%%%%%%%%%%%%%%%%%%%%%%%%%%%%%%%%%%%%%%%%%%%%%%%%%%%%

The bounds in Theorem~\ref{Th:Structured_bound} hold if the support of the system is only a
subset of $\calA$, for in that case, we may perturb the coefficients of the system to obtain
one whose support is exactly $\calA$ without decreasing the number of nondegenerate solutions.

When $n=d=\ell=2$, this bound for positive solutions is $83$.
This is already significant for the $(2,2)$-dense set~\eqref{Eq:triangle}.
Since the Kushnirenko bound for~\eqref{Eq:triangle} is 112,
a general system with support~\eqref{Eq:triangle} will have 112 complex solutions, at most 83
of which can be positive.

A generic system of $(d,\ell)$-dense fewnomials (as in~\eqref{Eq:near_polyhedral}),
 \begin{equation}\label{Eq:structured_system}
   f_1(x)\ =\ f_2(x)\ =\ \dotsb\ =\ f_n(x)\ =\ 0\,,
 \end{equation}
will have an invertible matrix of coefficients of the monomials 
$\{x^{w_i}\mid i=1,\dotsc,n\}$, and so we may solve~\eqref{Eq:structured_system} for these
monomials to get the equivalent system
 \begin{equation}\label{Eq:diagonal}
   x^{w_i}\ =\ \sum_{p\in d\ssimplex^\ell\cap\Z^\ell} a_{i,p} x^{\psi(p)}\,,
   \qquad\mbox{for}\quad i=1,\dotsc,n,.
 \end{equation}
For each $i=1,\dotsc,n$, define the degree $d$ polynomial in variables $y\in\R^\ell$,
\[
   \dcol{h_i(y)}\ :=\ \sum_{p\in d\ssimplex^\ell\cap\Z^\ell} a_{i,p} y^{p}\,.
\]

Following the notation in~\eqref{Eq:general_form}, we translate the set $\calA$ by $-v_0$,
which amounts to multiplying the equations~\eqref{Eq:structured_system}
and~\eqref{Eq:diagonal} by $x^{-v_0}$, and does not change their solutions.
Thus we may assume that $v_0=0$ so that $\psi$ is a linear map, and then let
$\calV:=\{v_1,\dotsc,v_\ell\}\subset\Z^n$ be the images of the standard basis vectors
under $\psi$. 
Then we have $x^{w_i}=h_i(x^{v_1},\dotsc,x^{v_\ell})$.
A linear relation among
the vectors in $\calV$ and $\calW$, 
 \begin{equation}\label{Eq:lin_reln}
   \sum_{m=1}^\ell b_m v_m\ +\ \sum_{i=1}^n c_i w_i\ =\ 0\,,
 \end{equation}
implies the multiplicative relation among the monomials
\[
   \prod_{m=1}^\ell (x^{v_m})^{b_m}\ \cdot\ \prod_{i=1}^n (x^{w_i})^{c_i}\ =\ 1\,.
\]
If we use~\eqref{Eq:diagonal} to first substitute $h_i(x^{v_1},\dotsc,x^{v_\ell})$
for $x^{w_i}$ for $i=1,\dotsc,n$ in this expression, 
and then substitute $y_m$ for $x^{v_m}$, for $m=1,\dotsc,\ell$, we obtain
\[
   \prod_{m=1}^\ell(y_m)^{b_m}\ \cdot\ \prod_{i=1}^n (h_i(y))^{c_i}\ =\ 1\,.
\]
Write $\beta=(b_1,\dotsc,b_\ell)$ for the vector of the coefficients of $\calV$
in~\eqref{Eq:lin_reln} and $\gamma=(c_1,\dotsc,c_n)$ for the vector of the coefficients of
$\calW$. 
Then we may write the left hand side of this last expression compactly as
$y^\beta\cdot h(y)^\gamma$.  
  
Now suppose that $(\beta_j,\gamma_j)\in\Z^\ell\oplus\Z^n$ for $j=1,\dotsc,\ell$ is a basis for
the $\Z$-module of linear relations among the vectors in $\calV\cup\calW$.
Then the system
 \begin{equation}\label{Eq:str_Gale}
    y^{\beta_j}\ \cdot\ h(y)^{\gamma_j}
    \ =\ 1\,,\qquad j=1,\dotsc,\ell
 \end{equation}
is a \demph{$(d,\ell)$-dense Gale system} dual to the original
system~\eqref{Eq:structured_system} of polynomials. 

%%%%%%%%%%%%%%%%%%%%%%%%%%%%%%%%%%%%%%%%%%%%%%%%%%%%%%%%%%%%%%%%%%%%%%%%%%%%
\begin{theorem}\label{Th:GaleDual}
 Let~$\eqref{Eq:structured_system}$ be a system of $(d,\ell)$-dense fewnomials
 and~$\eqref{Eq:str_Gale}$ be its corresponding dual $(d,\ell)$-dense Gale system. 
 Then the number of nondegenerate positive solutions to~$\eqref{Eq:structured_system}$ is equal
 to the number of nondegenerate positive solutions to~$\eqref{Eq:str_Gale}$ where
 $h_i(y)>0$ for each $i=1,\dotsc,n$. 

 If the exponents $\calA$ affinely span a sublattice of\/ $\Z^n$ of odd index and the relations  
 $(\beta^j,\gamma^j)\in\Z^\ell\oplus\Z^n$ for $j=1,\dotsc,\ell$ span a sub $\Z$-module  of odd index
 in the module of all linear relations, then the number of nondegenerate real solutions
 to~$\eqref{Eq:structured_system}$ is equal to the number of nondegenerate real solutions
 to~$\eqref{Eq:str_Gale}$ in $(\R^\times)^\ell$ where no $h_i(y)$ vanishes.
\end{theorem}
%%%%%%%%%%%%%%%%%%%%%%%%%%%%%%%%%%%%%%%%%%%%%%%%%%%%%%%%%%%%%%%%%%%%%%%%%%%%

This follows from Theorem~\ref{Th:Str_Gale} on  Gale duality for $(d,\ell)$-dense fewnomials. 
Thus we may prove Theorem~\ref{Th:Structured_bound} by establishing bounds for
$(d,\ell)$-dense Gale systems.\smallskip

Let us consider an example of this duality.
%%%%%%%%%%%%%%%%%%%%%%%%%%%%%%%%%%%%%%%%%%%%%%%%%%%%%%%%%%%%%%%%%%%%%%%%%%
%                     2                                    4
%      27         16 t        2         4       4  2   16 t
% f:= ---- + 31 - ----- - 16 t  u + 40 t  - 16 t  u  - -----
%       5           u                                    2
%      t                                                u
%
%                                                2      4
%               4  2       2        4        32 t    5 t
% g:= 12 t + 5 t  u  - 32 t  u + 6 t  + 40 - ----- + ----
%                                              u       2
%                                                     u
%%%%%%%%%%%%%%%%%%%%%%%%%%%%%%%%%%%%%%%%%%%%%%%%%%%%%%%%%%%%%%%%%%%%%%%%%%
The system of Laurent polynomials,
 \begin{equation}\label{Eq:Laurent_system}
  \begin{array}{rcl}
   \Color{1 0 1 0.4}{f}\ :=\ 
      27t^{-5}\ +\ 31 - 16t^2u - 16t^2u^{-1} -16t^4u^2+ 40t^4-16t^4u^{-2}&=& 0\,,\\
   \Color{0 1 1 0.3}{g}\ :=\ 
   12t\hspace{15.5pt}+\ 40 -32t^2u -32t^2u^{-1} 
      +\ \,5t^4u^2+\ \,6t^4+\ \,5t^4u^{-2}&=& 0\,,\rule{0pt}{14pt}
  \end{array}
 \end{equation}
has 36 complex solutions with nonzero coordinates,
ten of which are real and eight of which lie in the positive quadrant.
We show the curves $\Color{1 0 1 0.4}{f=0}$ and $\Color{0 1 1 0.3}{g=0}$ defined by the
polynomials~\eqref{Eq:Laurent_system}.
In the picture on the left, the horizontal scale has been exaggerated.
Its shaded region is shown on the right, where now the vertical scale is exaggerated.
 \begin{equation}\label{Eq:Laurent_Curves}
  %
  %  The pictures  poly_big.eps  poly_sm.eps wer made with poly.maple
  %  
  \raisebox{-60pt}{
  \begin{picture}(240,120)
    \put(0,0){\includegraphics[height=120pt]{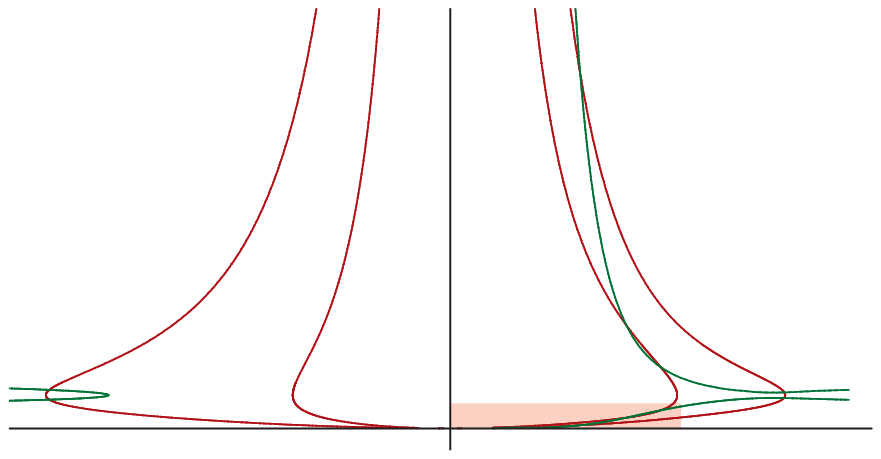}}

     \put(236,3){$t$} \put(121,110){$u$}

     \put(68,88){$\Color{0 1 1 0.3}{g}$}  \put(88,88){$\Color{0 1 1 0.3}{g}$} 
     \put(138,88){$\Color{0 1 1 0.3}{g}$} \put(180,44){$\Color{0 1 1 0.3}{g}$} 

    \put(1,49){$\Color{1 0 1 0.4}{f}$} \put(4,44){\Color{1 0 1 0.4}{\vector(0,-1){25}}}
    \put(228,14){$\Color{1 0 1 0.4}{f}$} 
  \end{picture}
   \qquad
   \begin{picture}(141,120)
    \put(133,6){$t$} \put(13,110){$u$}

    \put(0,0){\includegraphics[height=120pt]{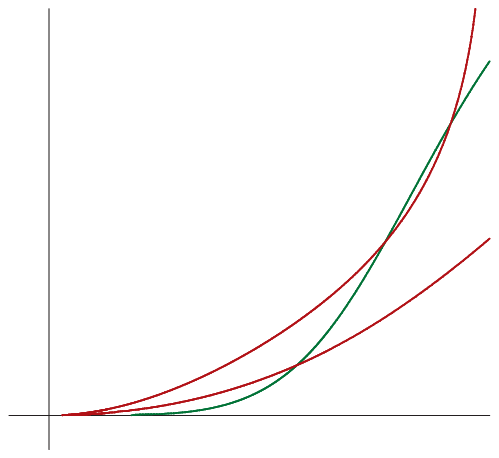}}
    \put(60,31){$\Color{0 1 1 0.3}{g}$}  \put(117,40){$\Color{0 1 1 0.3}{g}$} 
    \put(133,104){$\Color{1 0 1 0.4}{f}$} 
  \end{picture}}
 \end{equation}
Here are numerical approximations to the ten real solutions.
%
%  ( 0.618606849691311658,  0.093372960970825121),
%  ( 0.618606849691263868, 10.709738553888799964),
%  ( 0.839229221551949613,  0.322522650037381482),
%  ( 0.839229221551949613,  3.100557433358855776),
%  ( 1.002766398308815954,  0.542683175041158803),
%  ( 1.002766398308815954,  1.842695786402732767),
%  ( 1.590731244171416770,  0.911229703678423659),
%  ( 1.590731244171416770,  1.097418132840963344),
%  ( -1.91103170904055050,  0.863512949510988271),
%  ( -1.91103170904055050,  1.158060224303879904),
%
\begin{eqnarray*}
  &( 0.619,  0.093),\,
   ( 0.839,  0.326),\,
   ( 1.003,  0.543),\,
   ( 1.591,  0.911),\,
   (-1.911,  0.864),\,&\\
  &( 0.619, 10.71),\,
   ( 0.839,  3.101),\,
   ( 1.003,  1.843),\,
   ( 1.591,  1.097),\,
   (-1.911,  1.158)\,.&
\end{eqnarray*}
(The repetition in the $t$-coordinates is explained by the symmetry $u\mapsto u^{-1}$
of~\eqref{Eq:Laurent_system}.) 

The system~\eqref{Eq:Laurent_system} is a system of $(2,2)$-dense fewnomials, as we may see  
from its support.
 \[
  \begin{picture}(206,108)
   \put(0,0){\includegraphics{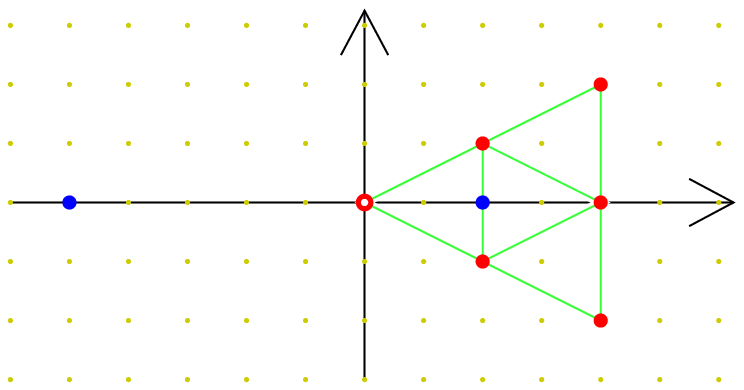}}
   \put(197,65){$t$}   \put(89,95){$u$}

   \put(129  ,76){$v_2$}  \put(131,28){$v_1$}
   \put(124.5,46){$w_2$}     \put( 14,45){$w_1$}
  \end{picture}
 \]
 If we solve~\eqref{Eq:Laurent_system} for the monomials $t^{-5}$ and $t$, we obtain
 \begin{equation}\label{Eq:solved_system}
  \begin{array}{rcl}
   t^{-5} &=&\frac{1}{27}(-\ 31 + 16t^2u + 16t^2u^{-1}+16t^4u^2- 40t^4+16t^4u^{-2})
     \ =:\ h_1(t^2u,t^2u^{-1})\\
   t\hspace{11pt}&=&\frac{1}{12}(-\ 40 +32t^2u + 32t^2u^{-1}  
               -\ \,5t^4u^2-\ \,6t^4-\ \hspace{1.5pt}5t^4u^{-2})
     \ =:\ h_2(t^2u,t^2u^{-1}) \,.\rule{0pt}{14pt}
  \end{array}
 \end{equation}
We convert this into a $(d,\ell)$-dense Gale system dual to~\eqref{Eq:Laurent_system}. 
First observe that 
 \begin{equation}\label{Eq:monoms}
  \begin{array}{rcl}
    \left(t^2u\right)^{\Red{1}}  \left(t^2u^{-1}\right)^{\Red{1}} 
    \left(t^{-5}\right)^{\Red{1}}  \left(t\right)^{\Red{1}} \hspace{6.5pt}
   &=&1\\
     \left(t^2u\right)^{\Red{2}}   \left(t^2u^{-1}\right)^{\Red{2}}
     \left(t^{-5}\right)^{\Red{1}}  \left(t\right)^{\Red{-3}} 
   &=&1\,.\rule{0pt}{14pt}
  \end{array}
 \end{equation}
We use the equations~\eqref{Eq:solved_system} to replace the monomials
$t^{-5}$ and $t$ in~\eqref{Eq:monoms} and then apply the substitutions
$x:=t^2u$ and $y:=t^2u^{-1}$ (so that $x^2=t^4u^2$,  $xy=t^4$, and $y^2=t^4u^{-2}$).
Then, after clearing denominators and
rearranging, we have 
%
%h1:= -(2*(2+x-2*y)*(2+y-2*x) - 1/4)*4/27:
%h2:= -(((x-2)+(y-2))^2/1^2 + ((x-2)-(y-2))^2/2^2 -6)/3:
%
 \begin{equation}\label{Eq:Gale}
  \begin{array}{rcl}
    x^{\Red{1}}  y^{\Red{1}}  h_1(x,y)^{\Red{1}} h_2(x,y)^{\Red{1}}
    \ -\ 1
   &=&0   \\
    x^{\Red{2}}  y^{\Red{2}} h_1(x,y)^{\Red{1}}
     \ -\   h_2(x,y)^{\Red{3}}
    &=&0   \,.\rule{0pt}{15pt}
  \end{array}
 \end{equation}
This system has 36 complex solutions, ten of which are real and eight of which lie in the
shaded region in the picture below where $x,y,h_1(x,y),$ and $h_2(x,y)$ are all positive.
 \begin{equation}\label{Eq:Gale_Curves}
  \raisebox{-90pt}{
   \begin{picture}(275,190)(-70,0)
    \put(0,0){\includegraphics{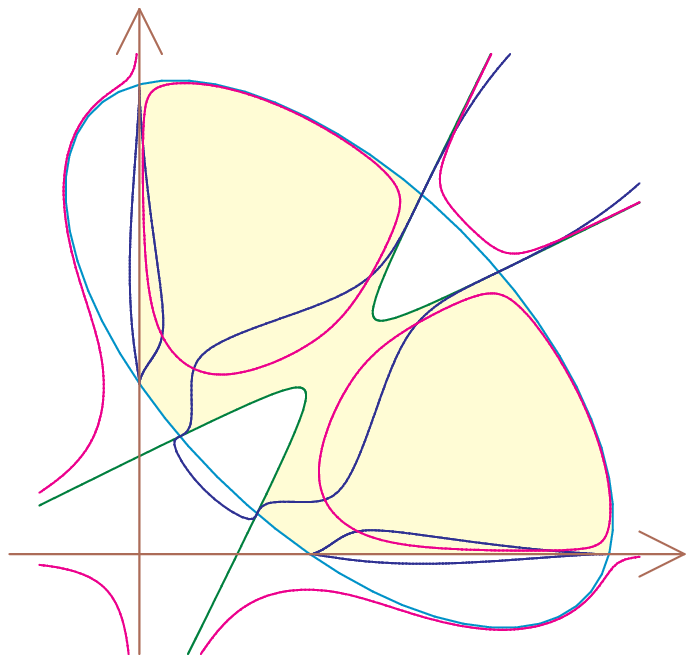}}
    \put(-62.5,65){$h_2(x,y)=0$}
     \put(0,68){\vector(1,0){45}}

    \put(-62.5,47.6){$h_1(x,y)=0$}    \put(141,150){$h_1(x,y)=0$}
      \put(155,145){\vector(-1,-1){44}}
     \put(0,50.5){\vector(4,-1){67}}

    \put(71,3){$x=0$} \put(-30,26){$y=0$}
     \put(69,6){\vector(-1,0){29}}
   \end{picture}
  }
 \end{equation}
The numbers of complex, real, and suitably positive solutions to the
two systems~\eqref{Eq:Laurent_system} and~\eqref{Eq:Gale} 
is a consequence of Theorem~\ref{Th:Str_Gale} on
structured Gale duality.
Here are numerical approximations to the ten real solutions of~\eqref{Eq:Gale} 
%
%(4.09834314459169519, 0.0357314450357013255),
%(2.18374023103512228, 0.2271545363840307911),
%(1.85290514949348037, 0.5456898838088880998),
%(0.54568988380888810, 1.8529051494934803705),
%(0.22715453638403079, 2.1837402310351222833),
%(2.77693525679463609, 2.3057992350030270159),
%(2.30579923500302707, 2.7769352567946360943),
%(4.22928480114469332, 3.1535857257801265684),
%(0.03573144503570132, 4.0983431445916951923),
%(3.15358572578012656, 4.2292848011446929069)
%
\begin{eqnarray*}
   &(4.229, 3.154),\,
    (4.098, 0.036),\,
    (2.777, 2.306),\,
    (2.184, 0.227),\,
    (1.853, 0.546),\,&
   \\
   &(3.154, 4.229),\,
    (0.036, 4.098),\,
    (2.306, 2.777),\,
    (0.227, 2.184),\,
    (0.546, 1.853)\,.&
\end{eqnarray*}

We remark that there is no relation between the two pairs of curves
in~\eqref{Eq:Laurent_Curves} and~\eqref{Eq:Gale_Curves}.
Gale duality only
asserts a scheme-theoretic equality between the points of intersection of each pair of curves.

%%%%%%%%%%%%%%%%%%%%%%%%%%%%%%%%%%%%%%%%%%%%%%%%%%%%%%%%%%%%%%%%%%%%%%%%%%%%
%
\section{Gale Duality for $(d,\ell)$-dense fewnomials}\label{S:gale}

Gale duality~\cite{BS_Gale} asserts that a system of $n$
polynomials in $n$ variables involving a total of $1{+}k{+}n$ distinct monomials is
equivalent to a system of $k$ rational functions of a particular form in the complement of
an arrangement of $k{+}n$ hyperplanes in $\R^k$.
A modification of Gale duality asserts that a system of $(d,\ell)$-dense fewnomials is
equivalent to a system $\ell\; (\leq k)$ rational functions in the complement $\calM(\R)$ of the 
coordinate axes of $\R^\ell$ and of $n$ degree $d$ hypersurfaces.
We will call such a system a $(d,\ell)$-dense Gale system.
Write $\dcol{\T}$ for the non-zero complex numbers, $\C^\times$.

%%%%%%%%%%%%%%%%%%%%%%%%%%%%%%%%%%%%%%%%%%%%%%%%%%%%%%%%%%%%%%%%%%%%%%%%%%%%
%
\subsection{$(d,\ell)$-dense polynomials}
\label{S:nearpolyhedral}

Suppose that $\calA=\psi(d\simplex^\ell\cap\Z^\ell)\cup\calW$ is $(d,\ell)$-dense and
that it affinely spans $\Z^n$.
Translating $\calA$ by $\psi(0)$ if necessary, we may assume that $\psi$ is linear.
Write $\calV=\{v_1,\dotsc,v_\ell\}\subset\Z^n$ for the images under $\psi$ of the standard
basis vectors of $\Z^\ell$ and list the elements of $\calW$ as $\{w_1,\dotsc,w_n\}$.
Consider the map
\[
  \begin{array}{rclcl}
   \varphi&\colon&\T^n&\longrightarrow&\T^\ell\times\T^n\\
   &&x&\longmapsto&(x^{v_1},\dotsc,x^{v_\ell},\,x^{w_1},\dotsc,x^{w_n})\,.
   \rule{0pt}{14pt}
  \end{array}
\]

Write $y=(y_1,\dotsc,y_\ell)$ for the coordinates of the first factor $\T^\ell$ and
$z=(z_1,\dotsc,z_n)$ for the coordinates of the second factor $\T^n$.
A polynomial with support $\calA$ has the form
 \begin{eqnarray}
  f&=& \sum_{\substack{\lambda=(\lambda_1,\dotsc,\lambda_\ell)\\|\lambda|=d}}
         a_\lambda x^{\lambda_1 v_1}\dotsb x^{\lambda_\ell v_\ell}\ +\ 
        \sum_{i=1}^n a_ix^{w_i}\nonumber\\
   &=& \varphi^*\Bigl(\sum_{|\lambda|=d} a_\lambda y_1^{\lambda_1}\dotsb y_\ell^{\lambda_\ell}\ +\ 
        \sum_{i=1}^n a_i z_i\Bigr) 
  \ =\  \varphi^*(h(y) + \Lambda(z))\,,\label{Eq:pullback}
 \end{eqnarray}
the pullback along $\varphi$ of a polynomial $h$ of degree $d$ and a linear form $\Lambda$. 

%%%%%%%%%%%%%%%%%%%%%%%%%%%%%%%%%%%%%%%%%%%%%%%%%%%%%%%%%%%%%%%%%%%%%%%%%%%%
%
\subsection{Proof of Theorem~\ref{Th:GaleDual}}

Let $\calA$ and $\calW$ be as in the previous subsection.
Suppose that $f_1(x)=\dotsb=f_n(x)=0$ is a system of $(d,\ell)$-dense fewnomials with
support $\calA$.
By~\eqref{Eq:pullback}, there exist polynomials $h_1(y),\dotsc,h_n(y)$ of degree $d$ in
the variables $y=(y_1,\dotsc,y_\ell)$ and linear forms $\Lambda_1(z),\dotsc,\Lambda_n(z)$
in variables $z=(z_1,\dotsc,z_n)$ such that
\[
   f_i(x)\ =\ \varphi^*( h_i(y) + \Lambda_i(z))
   \qquad i=1,\dotsc,n\,.
\]
Since we wish to enumerate non-degenerate solutions, we may assume that the polynomials
$h_i(y),\Lambda_i(z)$ are generic, for perturbing the coefficients of the $f_i$ can only
increase their number of non-degenerate solutions.

Thus we may assume that $\Lambda_1(z),\dotsc,\Lambda_n(z)$ are linearly independent.
Replacing the polynomials $f_1,\dotsc,f_n$ by appropriate linear combinations, we may
assume that $\Lambda_i(z)=-z_i$ for each $i$.
Then our system becomes
 \begin{equation}\label{Eq:graph_pullback}
   \varphi^*(h_1(y)-z_1)\ =\ 
   \varphi^*(h_2(y)-z_2)\ =\ \dotsb\ =\ 
   \varphi^*(h_n(y)-z_n)\ =\ 0\,.
 \end{equation}

If we define $H\subset\C^\ell\times\C^n$ by the equations
\[
   H\ =\ \{(y,z)\in\C^\ell\times\C^n\mid 
      z_1=h_1(y),\dotsc,z_n=h_n(y)\}\,,
\]
then our system~\eqref{Eq:graph_pullback} has the alternative geometric description
as $\varphi^*(H)$.
Since $\Z\calA=\Z^n$, $\varphi$ is an isomorphism onto its image, and we deduce the
following lemma.

\begin{lemma}
 The system~$\eqref{Eq:graph_pullback}$ is isomorphic to the intersection
 $\varphi(\T^n)\cap H$ in $\C^\ell\times\C^n$.
\end{lemma}

This is the first step in Gale duality for $(d,\ell)$-dense fewnomials.
For the second step observe that $H$ is isomorphic to $\C^\ell$, as it
is the graph of the function $\C^\ell\to\C^n$ given by $y\mapsto(h_1(y),\dotsc,h_n(y))$.
Let $\Psi\colon\C^\ell\to H$ be the isomorphism between $\C^\ell$ and this graph.
Then the system~\eqref{Eq:graph_pullback} is equivalent to
\[
   \varphi(\T^n)\cap \Psi(\C^\ell)\,.
\]
We determine the equations in $\T^\ell\times\T^n$ that define $\varphi(\T^n)$.

For $\beta=(b_1,\dotsc,b_\ell)\in\Z^\ell$ and $\gamma=(c_1,\dotsc,c_n)\in\Z^n$, let
\[
   y^\beta\cdot z^\gamma\ =\ y_1^{\beta_1}\dotsb y_\ell^{\beta_\ell}
    \cdot z_1^{c_1}\dotsb z_n^{c_n}\,.
\]
We similarly write $h(y)^\gamma$ for $h_1(y)^{c_1}\dotsb h_n(y)^{c_n}$.

Suppose that $\calB\subset\Z^{\ell}\oplus\Z^n$
is a basis for the $\Z$-linear relations among the exponent vectors $\calV\cup\calW$.
As $\calA$ spans $\R^n$, so does $\calV\cup\calW$, and so $\calB$ consists of
$\ell$ vectors, $\{(\beta_1,\gamma_1),\dotsc,(\beta_\ell,\gamma_\ell)\}$.
Then the image
$\varphi(\T^{n})\subset \T^\ell\times \T^n$ is
the subtorus defined by 
 \begin{equation}\label{Eq:sGale}
   y^{\beta_j}\cdot  z^{\gamma_j}\ =\ 1 \qquad\mbox{for } j=1,\dotsc,\ell\,.
 \end{equation}
%

%%%%%%%%%%%%%%%%%%%%%%%%%%%%%%%%%%%%%%%%%%%%%%%%%%%%%%%%%%%%%%%%%%%%%%%%%%%%%%%%%
\begin{proposition}\label{prop:npGale}
  The pullback of $\varphi(\T^{n})\cap H$ along the map
  $\Psi$ is the system
 \begin{equation}\label{Eq:npmf} 
   y^{\beta_j} \cdot h(y)^{\gamma_j}\ =\ 1
    \qquad\mbox{for } j=1,\dotsc,\ell\,.
 \end{equation}
  This is well-defined in $\C^{\ell}$ in the complement $\calM(\C)$ of the 
  coordinate planes and the hypersurfaces $h_i(y)=0$ for $i=1,\dotsc,n$.
\end{proposition}
%%%%%%%%%%%%%%%%%%%%%%%%%%%%%%%%%%%%%%%%%%%%%%%%%%%%%%%%%%%%%%%%%%%%%%%%%%%%%%%%%

We may now state our main theorem on structured Gale duality.
The \demph{saturation} of a submodule $\calB\subset\Z^{\ell}\oplus\Z^n$ is the set
$(\calB\otimes_\Z\R) \cap \Z^{\ell}\oplus\Z^n$ of integer points in its linear span.

%%%%%%%%%%%%%%%%%%%%%%%%%%%%%%%%%%%%%%%%%%%%%%%%%%%%%%%%%%%%%%%%%%%%%%%%%%%%%%%%%
\begin{theorem}\label{Th:Str_Gale}
  Suppose that $\calA,\calW,\psi$, and $\calV$ are as above and that $\calA$ 
  spans $\Z^{n}$.
  Then the solution set to~\eqref{Eq:graph_pullback} in $\T^{n}$
  is scheme-theoretically isomorphic to the solution set of the system of
  rational functions~\eqref{Eq:npmf} defined  in $\T^\ell$ in the
  complement of the hypersurfaces $h_i(y)=0$, for $i=1,\dotsc,n$.

  If the coefficients of the polynomials $f_i$ are real, then so are those of $h_i(y)$.
  If the span of $\calA$ has odd index in $\Z^{n}$ and the integer span of the
  exponents $\calB$ has odd index in its saturation, then the analytic subscheme 
  defined in $(\R^\times)^{n}$ by~\eqref{Eq:graph_pullback} is isomorphic to the analytic
  subscheme defined by~\eqref{Eq:npmf} in the complement of the
  hypersurfaces $h_i(y)=0$ in $(\R^\times)^\ell$.

  If now the exponents $\calA$ only span a full rank sublattice of
  $\Z^{n}$ and the exponents $\calB$ only span a full rank sublattice of
  the module  of linear relations among $\calV\cup\calW$, then the analytic subscheme of\/ 
  $\R_{>0}^{n}$ defined by~\eqref{Eq:graph_pullback} is isomorphic to the analytic subscheme 
  defined by~\eqref{Eq:npmf} in the subset of\/ $\R_{>0}^\ell$ defined by $h_i(y)>0$ for
  $i=1,\dotsc,n$. 
\end{theorem}
%%%%%%%%%%%%%%%%%%%%%%%%%%%%%%%%%%%%%%%%%%%%%%%%%%%%%%%%%%%%%%%%%%%%%%%%%%%%%%%%%

%%%%%%%%%%%%%%%%%%%%%%%%%%%%%%%%%%%%%%%%%%%%%%%%%%%%%%%%%%%%%%%%%%%%%%%%%%%%%%%%%
\begin{proof}
 The first statement concerning complex solutions is immediate from
 Proposition~\ref{prop:npGale} and the observation that the system~\eqref{Eq:graph_pullback} is
 the pullback of the intersection
 $\varphi(\T^n)\cap H$  along the  
 map $\varphi$, if we know that the map $\varphi$ is injective.
 Since $\calA$ spans $\Z^{n}$, so does $\calV\cup\calW$, and the map 
 $\varphi$ is injective.

 As the affine span of $\calA$ has odd index in  $\Z^{n}$, the map
 $\varphi$ is injective on $(\R^\times)^{n}$.
 As $\Z\calB$ has odd index in its saturation, the equations  
 $y^\beta\cdot z^\gamma=1$ for $(\beta,\gamma)\in\calB$ define the image
 $\varphi((\R^\times)^{n})$ in the real torus
 $(\R^\times)^\ell\times(\R^\times)^n$. 
 These facts in turn imply the second statement.
 
 Similarly, the hypotheses of the third statement imply the same facts about the 
 positive part of the real torus, $\R_{>0}^{n}$.
 Observing that the subset of $\R_{>0}^\ell$ defined by $h_i(y)>0$ for $i=1,\dotsc,n$ is the
 pullback of $\R_{>0}^\ell\times\R_{>0}^n$ under the map $\Psi$ completes the proof.
\end{proof} 
%%%%%%%%%%%%%%%%%%%%%%%%%%%%%%%%%%%%%%%%%%%%%%%%%%%%%%%%%%%%%%%%%%%%%%%%%%%%%%%%%

%%%%%%%%%%%%%%%%%%%%%%%%%%%%%%%%%%%%%%%%%%%%%%%%%%%%%%%%%%%%%%%%%%%%%%%%%%%%
%
\section{Bounds for $(d,\ell)$-dense Gale systems}\label{S:bounds}

By Theorem~\ref{Th:GaleDual}, Theorem~\ref{Th:Structured_bound} follows from 
bounds for $(d,\ell)$-dense Gale systems, which we give and prove  below.

Let $d,\ell,n$ be positive integers and $y_1,\dots,y_\ell$ be indeterminates.
Suppose that $h_i(y)$ for $i=1,\dots,n$ are generic degree $d$ polynomials.
Define 
 \begin{eqnarray*}
   \dcol{\Delta}&:=&\{y\in\R^\ell_{>0}\mid h_i(y)>0\ \mbox{for}\ i=1,\dots,n\}\,\qquad\mbox{and}\\
   \dcol{\calM(\R)}&:=&\{y\in(\R^\times)^\ell\mid h_i(y)\neq0\ \mbox{for}\ i=1,\dots,n\}\,,
 \end{eqnarray*}
and write $\calM(\C)$ for the complexification of $\calM(\R)$.

%%%%%%%%%%%%%%%%%%%%%%%%%%%%%%%%%%%%%%%%%%%%%%%%%%%%%%%%%%%%%%%%%%%%%%
\begin{theorem}\label{TH:main}
 With these definitions, suppose that 
 $\beta_1,\dotsc,\beta_\ell\in\Z^\ell$ and $\gamma_1,\dotsc,\gamma_\ell\in\Z^n$
 are vectors such that 
 $\calB:=\{(\beta_1,\gamma_1),\dotsc,(\beta_\ell,\gamma_\ell)\}\subset\Z^\ell\oplus\Z^n$
 are linearly independent.
 Then the number of solutions to
 \begin{equation}\label{Eq:Gale_Sy}
   y^{\beta_j}\cdot h(y)^{\gamma_{j}}\ =\ 1\,,
    \qquad\mbox{for}\quad j\ =\ 1,\dotsc,\ell\,,
 \end{equation}
 in the positive region $\Delta$ is less than
\[
   \frac{e^2{+}3}{4}\, 2^{\binom{\ell}{2}}\, n^\ell \cdot d^\ell\,.
\]
 If the integer span of $\calB$ has odd index in its saturation,
 then the number of solutions in $\calM(\R)$ is less
 than 
\[
   \frac{e^4{+}3}{4}\, 2^{\binom{\ell}{2}}\, n^\ell \cdot d^\ell\,.
\]
\end{theorem}
%%%%%%%%%%%%%%%%%%%%%%%%%%%%%%%%%%%%%%%%%%%%%%%%%%%%%%%%%%%%%%%%%%%%%%

We will deduce these bounds from several lemmata which we now formulate.
Their proofs are given in subsequent subsections.

For a vector $\alpha$, let $\alpha^\pm$ be the coordinatewise maximum of $\pm\alpha$ and $0$ so
that $\alpha^\pm$ is nonnegative, and $\alpha=\alpha^+-\alpha^-$.
Hence $(1,-2)^+=(1,0)$ and $(1,-2)^-=(0,2)$.
Set 
 \begin{equation}\label{Eq:form_g}
   \dcol{g_k(y)}\ :=\  y^{2\beta_k^+}h(y)^{2\gamma_k^+}-y^{2\beta_k^-}h(y)^{2\gamma_k^-}\,.
 \end{equation}
 Then $g_k(y)=0$ for $y\in\calM(\C)$ if and only if
 \begin{equation}\label{Eq:form_prod}
   \prod_{i=1}^\ell (y_i^2)^{\beta_{k,i}}\cdot
   \prod_{i=1}^n (h_i(y)^2)^{\gamma_{k,i}}\ =\ 1
 \end{equation}
 Notice that the system 
\[
   g_1(y)\ =\ g_2(y)\ =\ \dotsb\ =\ g_\ell(y)\ =\ 0\,,
\]
 is equivalent in $\Delta$ to the system~\eqref{Eq:Gale_Sy} and in $\calM(\R)$, it
 contains the system~\eqref{Eq:Gale_Sy} as a subsystem.
 We will bound the number of solutions to this expanded system in $\Delta$ and in $\calM(\R)$ to obtain
 our bounds for the system~\eqref{Eq:Gale_Sy} in Theorem~\ref{TH:main}. 

We state two important reductions.
%%%%%%%%%%%%%%%%%%%%%%%%%%%%%%%%%%%%%%%%%%%%%%%%%%%%%%%%%%%%%%%%%%%%%%
\begin{reduction}\label{R:reduction}
 It suffices to prove Theorem~$\ref{TH:main}$ under the following additional assumptions.
\begin{enumerate}
 \item
    For each $j=1,\dotsc,\ell$, the set
  $\dcol{\mu_j}\subset\calM(\C)$ defined by the equations
 \[
    \dcol{y^{2\beta_k}\cdot h(y)^{2\gamma_k}}\ :=\ 
    \prod_{i=1}^\ell (y_i^2)^{\beta_{k,i}}\cdot
    \prod_{i=1}^n (h_i(y)^2)^{\gamma_{k,i}}\ =\ 1\qquad
    k=1,\dotsc,j\,,
 \]
  is smooth and has codimension $j$.
  This condition holds for all sufficiently generic polynomials $h_i(y)$ of degree $d$.
 
 \item 
 For each $k=1,\dotsc,\ell$ define 
 $b_k:=  \beta_{k,1}+\dotsb+\beta_{k,\ell} + d(\gamma_{k,1}+\dotsb+\gamma_{k,n})$.
 Then every minor of the $\ell\times(1+\ell+n)$ matrix whose $k$th row is
 $(-b_k,\beta_k,\gamma_k)$ is nonzero.
\end{enumerate}
\end{reduction}
%%%%%%%%%%%%%%%%%%%%%%%%%%%%%%%%%%%%%%%%%%%%%%%%%%%%%%%%%%%%%%%%%%%%%%

We establish these reductions in Subsection~\ref{S:reduction}.

Our bounds are based on an induction which comes from 
the Khovanskii-Rolle Theorem, or more precisely, the induction is based on a modified form
which was used in~\cite{BS_mixed},
and which ensures that the hypotheses in subsequent lemmata hold.
See~\cite{BS_Khovanskii} and~\cite{BS_mixed} for more discussion.
For $D$ equal to either $\calM(\R)$ or its positive chamber $\Delta$ and $C$ an algebraic curve in
$D$, let \dcol{$\ubc_D(C)$} be the number of noncompact components of $C$.
Write  \dcol{$V_D(f_1,\dotsc,f_\ell)$} for the common zeroes in $D$ of functions 
$f_1,\dotsc,f_\ell$.

%%%%%%%%%%%%%%%%%%%%%%%%%%%%%%%%%%%%%%%%%%%%%%%%%%%%%%%%%%%%%%%%%%%%%
\begin{lemma}[Modified Khovanskii-Rolle Theorem]\label{L:KhRo}
 There exist polynomials $G_1,G_2,\dotsc,G_\ell$ where $G_{\ell-j}(y)$ is a generic
 polynomial with degree $2^jn\cdot d$ such that the following hold.
\begin{enumerate}
\item The system
 \[
      g_1(y)\ =\ \dotsb\ =\ g_j(y)\ =\
      G_{j+1}(y)\ =\ \dotsb\ =\ G_\ell(y)\ =\ 0\,,
 \]
  has only nondegenerate solutions in $\calM(\C)$ and the system
  \begin{equation}\label{Eq:C_j}
      g_1(y)\ =\ \dotsb\ =\ g_{j-1}(y)\ =\
      G_{j+1}(y)\ =\ \dotsb\ =\ G_\ell(y)\ =\ 0\,,
  \end{equation}
  ($g_j$ is omitted) defines a smooth curve in $\calM(\C)$.  
\item  Let $C_j\subset\calM(\R)$ be the solutions to~\eqref{Eq:C_j} in $\calM(\R)$, which is a smooth
  real algebraic curve.
  For $D$ be equal to either of $\calM(\R)$ or $\Delta$, we have
 \[
    |V_D(g_1,\dotsc,g_j,G_{j+1},\dotsc,G_\ell)|\ 
    \leq\ |V_D(g_1,\dotsc,g_{j-1},G_j,\dotsc,G_\ell)|+ \ubc_D(C_j)\,.
 \]
\end{enumerate}
\end{lemma}
%%%%%%%%%%%%%%%%%%%%%%%%%%%%%%%%%%%%%%%%%%%%%%%%%%%%%%%%%%%%%%%%%%%%%

This implies the following estimate.
Let $D$ equal to either of $\calM(\R)$ or $\Delta$.
Then we have 
 \begin{equation}\label{Eq:Fundamental_Estimate}
   |V_D(g_1,\dotsc,g_\ell)|\ \leq\  |V_D(G_1,\dotsc,G_\ell)| + 
     \ubc_D(C_1)+\dotsb+\ubc_D(C_\ell)\,.
 \end{equation}

Our next lemma estimates these quantities.

%%%%%%%%%%%%%%%%%%%%%%%%%%%%%%%%%%%%%%%%%%%%%%%%%%%%%%%%%%%%%%%%%%%%%
\begin{lemma}\label{L:estimates}
 We have 
 \begin{enumerate}
  \item $V_\Delta(G_1,\dotsc,G_\ell)\leq V_\calM(G_1,\dotsc,G_\ell)\leq 
    2^{\binom{\ell}{2}}n^\ell\cdot d^\ell$,
  \item \rule{0pt}{20pt}${\displaystyle\ubc_\Delta(C_{j})\leq 
          \tfrac{1}{2} 2^{\binom{\ell-j}{2}} n^{\ell-j}
          \tbinom{1+\ell+n}{j}}\cdot d^{\ell}$.
  \item \rule{0pt}{20pt}${\displaystyle\ubc_\calM(C_{j})\leq 
          \tfrac{1}{2} 2^{\binom{\ell-j}{2}} 2^j n^{\ell-j}
          \tbinom{1+\ell+n}{j}}\cdot d^{\ell}$.
  \item \rule{0pt}{20pt}$ 2^{\binom{\ell-j}{2}} n^{\ell-j}\tbinom{1+\ell+n}{j}\leq 
       \frac{1}{2} \frac{2^j}{j!} \cdot 2^{\binom{\ell}{2}}n^\ell$.
 \end{enumerate}
\end{lemma}
%%%%%%%%%%%%%%%%%%%%%%%%%%%%%%%%%%%%%%%%%%%%%%%%%%%%%%%%%%%%%%%%%%%%%

Statement (1) follows from Lemma~\ref{L:KhRo} by B\'ezout's Theorem, as 
$2^{\binom{\ell}{2}}n^\ell d^\ell$ is the product of the degrees of the polynomials
$G_1,\dotsc,G_n$.
Statement (4) is from the proof of Lemma 3.5 of~\cite{BS_Khovanskii}.
We prove the other statements of Lemma~\ref{L:estimates} in Subsection~\ref{S:estimates}.

Lemma~\ref{L:estimates} and the estimate~\eqref{Eq:Fundamental_Estimate} give us the
estimate
\[    |V_\Delta(g_1,\dotsc,g_\ell)|\  \leq\ 
    2^{\binom{\ell}{2}}n^\ell\cdot d^\ell\ +\   \sum_{j=1}^\ell
    \frac{1}{4} \frac{2^{j}}{j!}\cdot 2^{\binom{\ell}{2}}n^\ell\cdot d^\ell
   \ \leq\  \Bigl(1\ +\ \frac{1}{4} \sum_{j=1}^\ell \frac{2^j}{j!}\Bigr)
     \cdot 2^{\binom{\ell}{2}}n^\ell\cdot d^\ell \,.
 \]
The sum is a partial sum of the power series for $e^2-1$, 
and so we obtain
\[
  |V_\Delta(g_1,\dotsc,g_\ell)|\ <\
  \frac{e^2+3}{4}2^{\binom{\ell}{2}}n^\ell\cdot d^\ell\,.
\]
The estimation for $|V_\calM(g_1,\dotsc,g_\ell)|$ is similar. 
Using Lemma~\ref{L:estimates}(3) for $\ubc_\calM(C_j)$, the corresponding sum is now a 
partial sum for $e^4-1$, and so we obtain
\[
  |V_\calM(g_1,\dotsc,g_\ell)|\ <\
  \frac{e^4+3}{4}2^{\binom{\ell}{2}}n^\ell\cdot d^\ell\,,
\]
which completes the proof of Theorem~\ref{TH:main}.
%%%%%%%%%%%%%%%%%%%%%%%%%%%%%%%%%%%%%%%%%%%%%%%%%%%%%%%%%%%%%%%%% 

%%%%%%%%%%%%%%%%%%%%%%%%%%%%%%%%%%%%%%%%%%%%%%%%%%%%%%%%%%%%%%%%%%%%%%%%%%%%%%%
%
\subsection{Proof of reductions}\label{S:reduction}

The Reduction~\ref{R:reduction}(1) will follow from Bertini's Theorem that a general linear
section of a smooth quasi-projective variety is smooth and of the expected dimension.
First, define $\G_j\subset\T^\ell\times\T^n$ to be the subtorus defined by the equations
\[
   y^{2\beta_k}z^{2\gamma_k}\ =\ 1\qquad\mbox{for}\quad k=1,\dots,j\,.
\]
As in Section~\ref{S:gale}, let $\Psi\colon\C^\ell\to\C^\ell\times\C^n$ be the map
defined by
\[
   \Psi\ \colon\ y\ \longmapsto\ (y,\,h_1(y),\dotsc,h_n(y))\,.
\]
Then $\mu_j=\Psi^{-1}(\G_j\cap \Psi(\C^\ell))$.
Since $\Psi$ is an isomorphism onto its image and $\G_j$ has codimension $j$, it suffices to
show that $\G_j\cap\Psi(\C^\ell)$ is transverse. 

But this follows because $\Psi(\C^\ell)$ is the pullback of a linear subspace $L$ along
the map
 \begin{eqnarray*}
  \C^\ell\times\C^n &\longrightarrow& \C^{\binom{\ell+d}{d}}\times\C^n\\
  (y,z)&\longmapsto&((y^\lambda\colon |\lambda|=d),\, z)\,.
 \end{eqnarray*}
The linear space $L$ is defined by the coefficients of the polynomials
in a system Gale dual to the system~\eqref{Eq:Gale_Sy}.
Choosing $L$ to be generic, we may apply Bertini's Theorem and deduce that $\mu_j$ is smooth and
of codimension $j$.
We also see that this may be accomplished by choosing the polynomials $h_i(y)$
to be sufficiently generic.

For the second reduction, observe that our equations~\eqref{Eq:form_g} and~\eqref{Eq:form_prod}
are equivalent to
 \begin{equation}\label{eq:f_k}
   f_k(y)\ :=\ 
   \sum_{m=1}^\ell \beta_{k,m}\log|y_m|\ +\ 
   \sum_{i=1}^n \gamma_{k,i}\log|h_i(y)|\ =\ 0\,,\qquad
   k=1,\dotsc,\ell\,,
 \end{equation}
in $\calM(\R)$.
We may perturb them by changing the coefficients $\beta_{k,m}$ and $\gamma_{k,i}$ without
increasing their numbers of nondegenerate solutions.
Thus we can satisfy Reduction~\ref{R:reduction}(2) with real exponents.
Since the rational numbers are dense in the real numbers, we may satisfy
Reduction~\ref{R:reduction}(2) with rational exponents.
Finally, by clearing denominators, we may assume the exponents are integral.

%%%%%%%%%%%%%%%%%%%%%%%%%%%%%%%%%%%%%%%%%%%%%%%%%%%%%%%%%%%%%%%%%%%%%%%%%%%%%%%
%
\subsection{Proof of Lemma \ref{L:KhRo}}\label{S:LhRo}

We will establish Lemma~\ref{L:KhRo} by downward induction on $j$.
The main step is provided by the Khovanskii-Rolle Theorem, which we present in the
simplified form in which we need it.

%%%%%%%%%%%%%%%%%%%%%%%%%%%%%%%%%%%%%%%%%%%%%%%%%%%%%%%%%%%%%%%%%%%%%%%%%%%5
\begin{theorem}[Khovanskii-Rolle]
 Let $f_1,\dots,f_\ell$ be smooth functions defined on a domain $D\subset\R^\ell$
 where 
\[
  f_1(y)\ =\ f_2(y)\ =\ \dotsb\ =\ f_{\ell-1}(y)\ =\ 0\,,
\]
 defines a smooth curve $C$ in $D$.
 Let 
 \[
    J\ :=\ J(f_1,\dots,f_{\ell})\ :=\ 
   \det\left( \frac{\partial f_i}{\partial y_j} \right)_{i,j=1,\dots,\ell}
 \]
 be the Jacobian determinants of $f_1,\dotsc,f_\ell$.
 If $V_D(f_1,\dotsc,f_{\ell-1},J)$ is finite and if $C$ has finitely many components in $D$,
 then  $V_D(f_1,\dotsc,f_\ell)$ is finite and we have
 \begin{equation}\label{Eq:KhRoStep}
   |V_D(f_1,\dots,f_{\ell})|\ \leq\ |V_D(f_1,\dots,f_{\ell-1},J)|+ \ubc_D(C)\,.
 \end{equation}
\end{theorem}
%%%%%%%%%%%%%%%%%%%%%%%%%%%%%%%%%%%%%%%%%%%%%%%%%%%%%%%%%%%%%%%%%%%%%%%%%%%5

 To see this, note that on $C$, the Jacobian $J$ is proportional to the evaluation of the
 differential of $f_\ell$ on a tangent vector to $C$.
 Given two consecutive solutions $a,b$ to $f_\ell=0$ along an arc of $C$, 
 The Jacobian will have different signs at $a$ and at $b$, and therefore will vanish at least
 once between $a$ and $b$.
\[
  \begin{picture}(282,130)(-40,0)
   \put(0,0){\includegraphics[height=120pt]{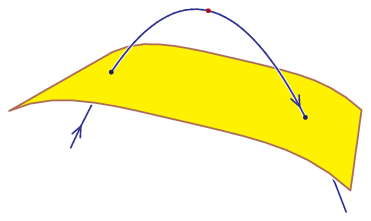}}
    \put(-40,20){$C=V_D(f_1,\dotsc,f_{\ell-1})$}
    \put(-20,65){$f_\ell=0$}
    \put(148,120){$V_D(f_1,\dotsc,f_{\ell-1},J)$}
    \put(146,122){\vector(-3,-1){30}}

    \put(19.2,110.3){$a$}   \put(26.2,109.3){\vector(1,-1){30}}
    \put(200,85.5){$b$}   \put(198.2,83.7){\vector(-1,-1){30}}
  \end{picture}
\] 
The estimate~\eqref{Eq:KhRoStep} follows as 
compact components of $C$ contain as many arcs connecting zeroes of $f_\ell$ as zeroes of
$f_\ell$, while noncompact components contain one arc fewer.

To deduce Lemma~\ref{L:KhRo}, we will iterate the Khovanskii-Rolle Theorem, showing
that the appropriate Jacobians have the claimed degrees and ensuring that its 
hypotheses are satisfied.

Observe that  Lemma~\ref{L:KhRo}(1) with $j=\ell$ holds by the assumptions we make in
Reduction~\ref{R:reduction}.
We prove Lemma~\ref{L:KhRo} by downward induction on $j=\ell,\dotsc,1$.
Specifically, we will assume that Statement (1) holds of some $j$ and then 
construct a polynomial $G_j$ such that (2) holds for $j$ and (1) holds for $j{-}1$.

To construct the polynomials $G_j$, we replace the rational functions 
$g_k(y)$ for $k=1,\dotsc,\ell$ in Lemma~\ref{L:KhRo}(2) by the logarithmic functions  
$f_k(y)$ for $k=1,\dotsc,\ell$ defined in~\eqref{eq:f_k}.
We may do this, as 
if $y\in\calM(\R)$, then $f_k(y)=0$ if and only if $g_k(y)=0$.

First, we need to determine the degrees of the Jacobians.

%%%%%%%%%%%%%%%%%%%%%%%%%%%%%%%%%%%%%%%%%%%%%%%%%%%%%%%%%%%%%%%%%%%%%%%%%%%%%%%%
\begin{lemma}\label{L:Newton}
 Let $1\leq j\leq \ell$ and suppose that $G_{j{+}1},\dotsc,G_\ell$ are polynomials, where $G_i$ has
 degree $2^{\ell-i}n\cdot d$, but is otherwise general, for each $i=1,\dotsc,\ell$.
 Then
 \begin{equation}\label{Eq:expression}
   \prod_{m=1}^\ell y_m\cdot\prod_{i=1}^n h_i(y)\cdot
   J(f_1,\dotsc,f_j,\, G_{j+1},\dotsc,G_\ell)
 \end{equation}
 is a polynomial with degree $2^{\ell-j}n\cdot d$.  
\end{lemma}
%%%%%%%%%%%%%%%%%%%%%%%%%%%%%%%%%%%%%%%%%%%%%%%%%%%%%%%%%%%%%%%%%%%%%%%%%%%%%%%%

We use this to deduce Lemma~\ref{L:KhRo}.
Suppose that we have polynomials $G_{j+1},\dotsc,G_\ell$ where $G_i$ is a generic 
polynomial with degree $2^{\ell-i}n\cdot d$, for each $i=j{+}1,\dotsc,\ell$,
and Lemma~\ref{L:KhRo}(1) holds for $j$.

Let $C_j\subset\calM(\R)$ be the smooth real algebraic curve defined by 
\[
   f_1(y)\ = \ \dotsb\ =\ f_{j-1}(y)\ \;=\;\ 
   G_{j+1}(y)\ =\ \dotsb\ =\ G_\ell(y)\ =\ 0\,.
\]
Let $\overline{G}_j$ be the product of the Jacobian 
$J(f_1,\dotsc,f_j,\, G_{j{+}1},\dotsc,G_\ell)$ with the polynomial 
\[
   \Upsilon(y)\ :=\ \prod_{m=1}^\ell y_m\cdot\prod_{i=1}^n h_i(y)\,.
\]
Then $\overline{G}_j$ is a polynomial with degree $2^{\ell-j}n\cdot d$, by
Lemma~\ref{L:Newton}. 
Since $\Upsilon$ does not vanish in $\calM(\R)$, the polynomial $\overline{G}_\ell$
and the Jacobian define the same sets, and we have
 \begin{multline}\label{Eq:KhRo_est}
   \qquad|V_D(f_1,\dotsc,f_j,\,G_{j+1},\dotsc,G_\ell)|\ \leq\ \\
   |V_D(f_1,\dotsc,f_{j-1},\,\overline{G}_j,\,G_{j+1},\dotsc,G_\ell)|\ +\ 
   \ubc_D(C_j)\,,\qquad
 \end{multline}
by the Khovanskii-Rolle Theorem.

Note however that we do not know if $\overline{G}_j$ is a generic polynomial 
with degree $2^{\ell-j}n\cdot d$, and in particular, we do not know if the
hypotheses of Lemma~\ref{L:KhRo}(1) for $j{-}1$ hold.
These hypotheses may be achieved by perturbing $\overline{G}_j$ to a nearby generic
polynomial $G_j$ with degree $2^{\ell-j}n\cdot d$.
To ensure that this perturbation does not destroy the estimate~\eqref{Eq:KhRo_est},
we only need to guarantee that the signs of $G_j$ and $\overline{G}_j$ are the same at
every point of $V_D(f_1,\dotsc,f_j,\,G_{j+1},\dotsc,G_\ell)$, but this will hold for all
sufficiently small perturbations, as there are only finitely many such points
and $\overline{G}_j$ is nonzero at each.
These conditions are equivalent to every point of
$V_D(f_1,\dotsc,f_j,\,G_{j+1},\dotsc,G_\ell)$ being nondegenerate, which
is ensured by the genericity of $G_{j{+}1},\dotsc,G_\ell$.
This completes the proof of Lemma~\ref{L:KhRo}.

%%%%%%%%%%%%%%%%%%%%%%%%%%%%%%%%%%%%%%%%%%%%%%%%%%%%%%%%%%%%%%%%%%%%%%%%%%%%%
\begin{proof}[Proof of Lemma~$\ref{L:Newton}$]
 Write $\partial_m$ for $\partial/\partial y_m$,
 and consider the expression~\eqref{Eq:expression}, writing the Jacobian in block form,
 with $j$ rows for $f_1,\dotsc,f_j$ and $\ell-j$ rows for $G_{j+1},\dotsc,G_\ell$.
 \begin{multline}\label{Eq:block}
   \qquad \prod_{m=1}^\ell y_m\cdot\prod_{i=1}^n h_i(y)\cdot
   \det\left(\, \begin{matrix}
    \bigl(\partial_m f_k(y)\bigr)_{k=1,\dotsc,j}^{m=1,\dotsc,\ell}
     \smallskip\\\hline\rule{0pt}{15pt}
    \bigl(\partial_m G_k(y)\bigr)_{k=j{+}1,\dotsc,\ell}^{m=1,\dotsc,\ell}
   \end{matrix}\,\right)\\
  \ =\ 
   \prod_{i=1}^n h_i(y)\cdot
   \det\left(\, \begin{matrix}
    \bigl(y_m\partial_m f_k(y)\bigr)_{k=1,\dotsc,j}^{m=1,\dotsc,\ell}
     \smallskip\\\hline\rule{0pt}{15pt}
    \bigl(y_m\partial_m G_k(y)\bigr)_{k=j{+}1,\dotsc,\ell}^{m=1,\dotsc,\ell}
   \end{matrix}\,\right)\,.\qquad 
 \end{multline}
 Laplace expansion along the first $j$ rows of the matrix on the right expresses its
 determinant as a sum of products of maximal minors of the two blocks.
 We will prove Lemma~\ref{L:Newton} by showing that each term in that sum is a polynomial with
 degree $2^{\ell-j}n\cdot d$.

 First, the lower block 
 $\bigl(y_m\partial_m G_k(y)\bigr)_{k=j{+}1,\dotsc,\ell}^{m=1,\dotsc,\ell}$
 is a matrix of polynomials whose entries in row $k$ are the toric derivatives 
 $y_m\partial_m G_k(y)$ of $G_k$.
 Thus every entry in row $k$ has degree $\deg(G_k)=2^{\ell-k}n\cdot d$, and
 therefore each minor has degree 
 \begin{equation}\label{Eq:lower_degree}
  2^{\ell-(j+1)}n\cdot d \ +\ \dotsb\ +\ 
  2n\cdot d \ +\ n\cdot d\ =\ 
  (2^{\ell-j}-1)n\cdot d\,.
 \end{equation}

 For the upper block, note that $y_m\partial_m  f_k(y)$ is
\[
  \beta_{k,m}\ +\ \sum_{i=1}^k \gamma_{k,i}y_m\partial_m \log|h_i(y)|\ =\ 
  \beta_{k,m}\ +\ \sum_{i=1}^n\gamma_{k,i} \frac{y_m \partial_m h_i(y)}{h_i(y)}\,. 
\]
 In particular, the upper block is a product of a $j\times(\ell+n)$ matrix and a
 $(\ell+n)\times \ell$ matrix,
\[
    \bigl(y_m\partial_m f_k(y)\bigr)_{k=1,\dotsc,j}^{m=1,\dotsc,\ell}\ =\ 
    \bigl(\, \beta_{k,q}\ \mid\ \gamma_{k,i}\,\bigr)\cdot
    \left(\begin{matrix} I_\ell\medskip\\ 
     y_m \partial_m h_i(y)/h_i(y)\end{matrix}\right)\ ,
\]
 where $I_\ell=(\delta_{q,m})$ is the $\ell\times\ell$ identity matrix.
 By the Cauchy-Binet formula, a $j\times j$ minor of $\bigl(y_m\partial_m f_k(y)\bigr)$ is a
 sum of products of $j\times j$ minors of the two matrices on the right.
 Consider now the product of $\prod_{i=1}^nh_i(y)$ with a term in this sum.

 The first matrix $\bigl(\beta_{k,q}\, \mid\, \gamma_{k,i}\bigr)$ contains constants, and
 a $j\times j$ minor of the second involves no more than $p:=\min\{n,j\}$ rows from
 its lower $n\times\ell$ block $(y_m\partial_mh_i(y)/h_i(y))$.
 This minor is a sum of $j!$ terms, each one a product of a constant and $p$ entries of the 
 matrix of the form $y_m\partial_mh_i(y)/h_i(y)$, for different rows $i$.
 Multiplying this term by $\prod_{i=1}^nh_i(y)$ will clear all 
 denominators and result in a product of $p$ terms of the form
 $y_m\partial_mh_i(y)$ and complementary $n{-}p$ terms of the form $h_i(y)$.
 Each of these terms has degree $d$, so each term coming from the
 expansion of this $j\times j$ minor has degree $n\cdot d$, and therefore
 the product of $\prod_{i=1}^nh_i(y)$ by each $j\times j$ minor of the upper block
 of~\eqref{Eq:block} will have degree $n\cdot d$.

 Together with~\eqref{Eq:lower_degree}, this completes the proof.
\end{proof}
%%%%%%%%%%%%%%%%%%%%%%%%%%%%%%%%%%%%%%%%%%%%%%%%%%%%%%%%%%%%%%%%%%%%%%%%%%%%%

%%%%%%%%%%%%%%%%%%%%%%%%%%%%%%%%%%%%%%%%%%%%%%%%%%%%%%%%%%%%%%%%%%%%%%%%%%%%%
%
\subsection{Proof of Lemma~\ref{L:estimates}}\label{S:estimates}

We only need to prove Statements (2) and (3) of Lemma~\ref{L:estimates}.
By Reduction~\ref{R:reduction} and Lemma~\ref{L:KhRo}, we may assume that the polynomials
$h_i(y)$ and $G_j(y)$ are generic given their degrees.

The complexification of the real curve $C_j$ is defined in $\C^\ell\supset\calM(\C)$ by
 \begin{equation}\label{Eq:curvej}
   g_1(y)\ =\ \dotsb\ =\ g_{j-1}(y)\ \ =\ \ 
   G_{j+1}(y)\ =\ \dotsb\ =\ G_\ell(y)\ =\ 0\,,
 \end{equation}
and it lies on the codimension $j{-}1$ subvariety $\mu_{j-1}\subset\calM(\C)$ defined by
\[
   g_1(y)\ =\ \dotsb\ =\ g_{j-1}(y)\ =\ 0\,.
\]
We bound the number of unbounded components of $C_j$ by first describing the points where
$\mu_{j-1}(\C)$ meets the boundary of $\calM(\C)$, then bound the number of real solutions to 
 \begin{equation}\label{Eq:bigG}
     G_{j+1}(y)\ =\ \dotsb\ =\ G_\ell(y)\ =\ 0
 \end{equation}
on these boundary points, and lastly by determining the number of components of $C_j$ incident
upon each such real solution.

To accomplish this, consider $\calM(\C)\subset\C^\ell$ as a subset of projective space
$\P^\ell$.
Its boundary $\partial\calM(\C):=\P^\ell\setminus\calM(\C)$
consists of the finite coordinate planes $y_m=0$ for $m=1,\dotsc,\ell$, the coordinate
plane at inifinty $y_0=0$, and the degree $d$ hypersurfaces $h_i(y)=0$ for $i=1,\dotsc,n$.
Strictly speaking, we must homogenize polynomials $h_i,g_j,G_k$ with respect to the coordinate
$y_0$ at infinity.
When working on an affine patch where $y_m\neq 0$, we de-homogenize them by setting
$y_m=1$. 

By our assumption that the polynomials $h_i$ were general, this boundary
$\partial\calM(\C)$ forms a divisor with normal crossings whose components are the coordinate
planes and the hypersurfaces $h_i(y)=0$.
The common zeroes of any $q$ of the polynomials $h_1,\dotsc,h_n$ and $j{-}q$ of the
coordinates $y_0,\dotsc,y_\ell$ is a smooth subvariety of codimension $j$, called a
\demph{codimension-$j$ stratum}.
The union of these $j$-fold intersections of the components of the boundary divisor
is called the \demph{codimension-$j$ skeleton} of $\partial\calM(\C)$.

%%%%%%%%%%%%%%%%%%%%%%%%%%%%%%%%%%%%%%%%%%%%%%%%%%%%%%%%%%%%%
\begin{lemma}
  The closure $\overline{\mu_{j-1}}$ meets $\partial\calM(\C)$ in a union of
  codimension-$j$ strata and in the neighborhood of a real point of\/
 $\overline{\mu_{j-1}}$ lying in the relative interior of a codimension-$j$ stratum,
 $\mu_{j-1}$ has one branch in each of the $2^j$ components of $\calM(\R)$
 incident on that point. 
\end{lemma}
%%%%%%%%%%%%%%%%%%%%%%%%%%%%%%%%%%%%%%%%%%%%%%%%%%%%%%%%%%%%%

%%%%%%%%%%%%%%%%%%%%%%%%%%%%%%%%%%%%%%%%%%%%%%%%%%%%%%%%%%%%%
\begin{proof}
 Since $\mu_{j-1}\subset\calM(\C)$ has codimension $j{-}1$ in $\P^\ell$, the
 intersection of its closure $\overline{\mu_{j-1}}$ with the 
 boundary divisor $\partial\calM(\C)$ will have codimension $j$ in $\P^\ell$.
 We prove the first part of the lemma by showing that this intersection consists of points
 lying within the codimension-$j$ skeleton of $\partial\calM(\C)$, and therefore is a union of
 components of the codimension-$j$ skeleton.

 Let $Y$ be a point of $\partial\calM(\C)$ that does not lie in the codimension-$j$ skeleton. 
 We show that $Y\not\in\overline{\mu_{j-1}}$.
 Since $\partial\calM(\C)$ is a divisor with normal crossings whose components are defined by
 the coordinates $y_i$ and forms $h_i$, at least one, but no more than
 $j{-}1$ of the coordinates $y_0,y_1,\dotsc,y_\ell$ and forms
 $h_1,\dotsc,h_n$ vanish at $Y$.
 Reordering the coordinates and forms if necessary, we may assume that the forms which
 vanish at $Y$ are among $h_1,\dotsc,h_q$ and the coordinates are among
 $y_{q+1},\dotsc,y_{j-1}$.
 Since the assertion about $Y$ is local, we may restrict to the affine patch $U$ where
 none of the remaining coordinates or forms vanish.
 
 The equations
 \begin{equation}\label{Eq:mujeqs}
   y^{2\beta_k} h(y)^{2\gamma_k}\ =\ 1\qquad k=1,\dotsc,j{-}1\,,
 \end{equation}
 define $\mu_{j-1}\subset\calM(\C)$.
 By Reduction~\ref{R:reduction}(2) on the homogenized exponent vectors,
 there is an integer linear combination of the first $j{-}1$
 rows of the matrix $(-b_k,\beta_k,\gamma_k)$ so that the columns corresponding to 
 $h_1,\dotsc,h_q$ and  $y_{q+1},\dotsc,y_{j-1}$ become diagonal.
 These same linear combinations transform the equations~\eqref{Eq:mujeqs} into equations
 of the form
 \begin{equation}\label{Eq:skeletondiagonal}
  \begin{array}{rcl}
   h_i(y)^{a_i} &=& y^{\alpha_i}\cdot h(y)^{\delta_i}\qquad i=1,\dotsc,q\,,\\
   y_i^{a_i} &=& y^{\alpha_i}\cdot h(y)^{\delta_i}\qquad
   i=q{+}1,\dotsc,j{-}1\,,\rule{0pt}{15pt} 
 \end{array}
 \end{equation}
 where $a_i> 0$ and the components of $\alpha_i$ in positions $q{+}1,\dotsc,j{-}1$
 vanish as do the components of $\delta_i$ in positions $1,\dotsc,q$.
 That is, $h_1,\dotsc,h_q$ and $y_{q+1},\dotsc,y_{j-1}$ do not appear on the right of these
 expressions. 
 
 Since the expressions in~\eqref{Eq:skeletondiagonal} are well-defined functions in $U$,
 the regular functions
\[
  \begin{array}{rl}
    h_i(y)^{a_i}\ -\ y^{\alpha_i}\cdot h(y)^{\delta_i}\quad &\mbox{for}\quad i=1,\dotsc,q\,,
    \quad\mbox{and}\\
   y_i^{a_i}\ -\ y^{\alpha_i}\cdot h(y)^{\delta_i}\quad&
   \mbox{for}\quad i=q{+}1,\dotsc,j{-}1\rule{0pt}{15pt}  
 \end{array} 
\]
 vanish on $\mu_{j-1}\cap U$, and hence on $\overline{\mu_{j-1}}\cap U$.
 However, these cannot all vanish at $Y$, for none of the functions 
 $y^{\alpha_i}h^{\delta_i}$ for $i=1,\dotsc,j{-}1$ vanish at $Y$, but at least one of
 $h_1^{a_1},\dotsc,h_q^{a_q}$, $y_{q+1}^{a_{q+1}},\dotsc,y_{j-1}^{a_{j-1}}$ vanishes at
 $Y$.

 To complete the proof, suppose that $Y\in\overline{\mu_{j-1}}(\R)\cap\partial\calM(\R)$ is a
 real point lying on a
 codimension-$j$ stratum of $\partial\calM(\C)$ but not on a stratum of larger codimension.
 Reordering the coordinates and functions if necessary and working locally, we may assume that 
 the polynomials $h_1(y),\dotsc,h_q(y)$ and coordinates $y_{q+1},\dotsc,y_{j}$ vanish at
 $Y$.
 Thus in the affine neighborhood $U$ of $Y$ where none of the other polynomials or coordinates
 vanish, $\mu_{j-1}$ is defined by equations of the form
 \begin{equation}\label{Eq:newdiagonal}
  \begin{array}{rcl}
   h_i(y)^{a_i} &=& y_j^{c_i}\cdot y^{\alpha_i}\cdot h(y)^{\delta_i}\qquad i=1,\dotsc,q\,,\\
   y_i^{a_i} &=&  y_j^{c_i}\cdot y^{\alpha_i}\cdot h(y)^{\delta_i}\qquad
   i=q{+}1,\dotsc,j\,,\rule{0pt}{15pt} 
 \end{array}
 \end{equation}
 where $a_i>0$ and none of $h_1,\dotsc,h_q$ and $y_{q+1},\dotsc,y_j$ appear in the expressions
 $y^{\alpha_i}\cdot h(y)^{\delta_i}$ for $i=1,\dotsc,j{-}1$.
 In fact, we must have $c_i>0$ as $Y\in\overline{\mu_{j-1}}$.
 In a neighborhood of $Y$ in $\R\P^\ell$, the complement $\calM(\R)$ has $2^j$ chambers given
 by the signs of the functions $h_1(y),\dotsc,h_q(y)$, $y_{q+1},\dotsc,y_j$.
 Since the exponents in~\eqref{Eq:newdiagonal} have every component even
 (this comes from the evenness of the  exponents in~\eqref{Eq:mujeqs}), there is a
 component of $\mu_{j-1}$ in each of these chambers.
\end{proof}
%%%%%%%%%%%%%%%%%%%%%%%%%%%%%%%%%%%%%%%%%%%%%%%%%%%%%%%%%%%%%

We complete the proof of Statements (2) and (3) of Lemma~\ref{L:estimates}.
We estimate the number of unbounded components of the curve $C_j$ by first bounding the
number of points where its closure in $\R\P^\ell$ meets the boundary divisor, and then 
bounding the number of components of $C_j$ incident upon each point.

The estimate for the number of points in $\overline{C_j}\cap\partial\calM(\R)$ is simply
the number of points in the codimension-$j$ skeleton where
 \begin{equation}\label{eq:Gs}
   G_{j+1}(y)\ =\ G_{j+2}(y)\ =\ \dotsb\ =\ G_\ell(y)\ =\ 0\,.
 \end{equation}
Consider a stratum where $q$ of the polynomials $h_i$ vanish and $j{-}q$ of the coordinates
$y_m$ vanish.
Since the polynomials $G_i$ and $h_i$ are general given their degrees, the number of
points on this stratum will be at most the product of these degrees, which is
\[   
   d^q \cdot 2^{\ell-j-1} nd\cdot 2^{\ell-j-2} nd \dotsb 2 nd \cdot nd \ =\ 
   2^{\binom{\ell-j}{2}} \cdot n^{\ell-j} d^{\ell-j} d^q\,.
\]
Since there are $\binom{\ell+1}{j-q}\cdot\binom{n}{q}$ such strata, the number of points
where $C_j$ meets the boundary is at most
 \begin{equation}\label{eq:cjbdry}
   2^{\binom{\ell-j}{2}} \cdot  n^{\ell-j}  d^{\ell-j}
    \cdot\ \sum_{q=0}^j \binom{\ell+1}{j-q}\binom{n}{q}\cdot d^q\ <\ 
   2^{\binom{\ell-j}{2}} \cdot  n^{\ell-j} \binom{1+\ell+n}{j}
    \cdot   d^\ell\ .
 \end{equation}
 
As the polynomials $G_i$ are general the variety defined by~\eqref{eq:Gs} is transverse to
the codimension-$j$ stratum, so there is at most one branch of $C_j$ in each component of
$\calM(\R)$ meeting such a point.
Thus the number~\eqref{eq:cjbdry} bounds the number of ends of components
of $C_j$ in $\Delta$, so it bounds twice the number of unbounded components of $C_j$ in
$\Delta$. 
For the bound in $\calM(\R)$, we multiply this by $2^j$, as there are $2^j$ components of
$\calM(\R)$ meeting each such point, and each component of $\calM(\R)$ contains at most one branch of
$C_j$. 

This completes the proof of the  Lemma~\ref{L:estimates}.

%%%%%%%%%%%%%%%%%%%%%%%%%%%%%%%%%%%%%%%%%%%%%%%%%%%%%%%%%%%%%%%%%%%%%%%%%%%%%
%
\section{Betti Number Bounds}\label{Sec:Betti}

Using the bound~\eqref{Eq:structured_bound} and stratified Morse theory for
a manifold with corners~\cite{GM88}, we prove the following theorem.

%%%%%%%%%%%%%%%%%%%%%%%%%%%%%%%%%%%%%%%%%%%%%%%%%%%%%%%%%%%%%%%%%%%%%%%%%%%
\begin{theorem}\label{betti}
 Let $X$ be a hypersurface in $\R^n_>$ defined by a $(d,\ell)$-dense fewnomial.
 Then
 \[
    b_*(X)\ <\ \frac{e^2{+}3}{4}\, 2^{\binom{\ell}{2}} \cdot d^\ell
    \cdot \sum_{i=0}^{n} \tbinom{n}{i}\, i^\ell\,.
 \]
\end{theorem}
%%%%%%%%%%%%%%%%%%%%%%%%%%%%%%%%%%%%%%%%%%%%%%%%%%%%%%%%%%%%%%%%%%%%%%%%%%%

Let $f(x_1,\dots,x_n)$ be a real Laurent polynomial with $(d,\ell)$-dense support
$\calA \subset \Z^n$ such that 
$X := \calV(f) \subset \R^n_>$ is a smooth hypersurface. 
By a logarithmic change of coordinates, $x_i=e^{z_i}$, we may work with exponential sums in  
$\R^n$  instead of polynomials in $\R^n_>$. 
Then the $(d,\ell)$-dense fewnomial 
$f = \sum_{\alpha \in \calA} c_{\alpha} x^{\alpha}$ becomes
the exponential sum
\[
 \varphi\ :=\ \sum_{\alpha \in \calA} c_{\alpha} e^{z \cdot \alpha}.
\]
In this way, the bounds~\eqref{Eq:BS} and~\eqref{Eq:structured_bound} for positive
solutions to fewnomial systems hold for real solutions to systems of exponential sums with the 
same exponents.  

Let $Z := \calV(\varphi) \subset \R^n$ be the hypersurface defined by $\varphi$, which is
homeomorphic to $X$. We will prove Theorem~\ref{betti} in these logarithmic coordinates
and with real exponents.

\begin{theorem}\label{bettilog}
 The sum of the Betti numbers of a hypersurface in $\R^n$ defined by an exponential sum
 whose exponent vectors are $(d,\ell)$-dense
 $\calA = \calW\, \cup\, \psi(d\simplex^\ell\cap\Z^\ell) \subset \R^n$ is at most
 \[
    \frac{e^2{+}3}{4}\, 2^{\binom{\ell}{2}} \cdot d^\ell
    \cdot \sum_{i=0}^{n} \tbinom{n}{i}\, i^\ell\,.
 \]
\end{theorem}

 An affine change of coordinates replaces $\calW$ with another
 set of independent vectors and replaces $\psi$ with another affine map but it
 does not change the $(d,\ell)$-dense structure. 
 We may thus assume that the vectors in $\calW$ are the standard unit basis vectors 
 in $\R^n$, and so $\varphi$ includes the coordinate exponentials $e^{z_i}$ for 
 $i = 1,\dots,n$. 
 Let $M \colon = (M_0,M_1,\dots,M_n)$ be a list of positive numbers and set
\[
 \Delta(M) \ :=\ \{z \in \R^n \mid z_i \geq -M_i,\, i = 1,\dots,n \quad \text{ and }
 \quad \sum_i z_i \leq M_0\}, 
\]
 which is a nonempty simplex. We will use stratified Morse theory to bound the Betti
 numbers of $Y := Z \cap \Delta_M$ when $M$ is general.

%%%%%%%%%%%%%%%%%%%%%%%%%%%%%%%%%%%%%%%%%%%%%%%%%%%%%%%%%%%%%%%%%%%%%%
\begin{theorem}\label{bettigeneral}
 For $M$ general, the sum of the Betti numbers of\/ $Y$ is at most
 \[
    \frac{e^2{+}3}{4}\, 2^{\binom{\ell}{2}} \cdot d^\ell
    \cdot \sum_{i=0}^{n} \tbinom{n}{i}\, i^\ell\,.
 \]
\end{theorem}
%%%%%%%%%%%%%%%%%%%%%%%%%%%%%%%%%%%%%%%%%%%%%%%%%%%%%%%%%%%%%%%%%%%%%%

Theorem~\ref{bettilog} is a consequence of Theorem~\ref{bettigeneral}. 
See Theorem $1'$ in~\cite{BS_Betti} for a detailed proof.

\begin{proof}[Proof of Theorem~$\ref{bettigeneral}$]
 Given positive numbers $M = (M_0,M_1,\dots,M_n)$, define affine hyperplanes in $\R^n$
 by
 \[
  H_0\ :=\ \{z \mid \sum_i z_i = M_0\} \quad \text{ and } \quad 
  H_i\ :=\ \{z \mid z_i = -M_i\}, \quad \text{ for } i = 1,\dots,n\,.
 \]
 For each proper subset $S \subset \{0,\dots,n\}$, define an affine linear subspace
 \[
  H_S\ :=\ \bigcap_{i \in S} H_i\,.
 \]
 Since each $M_i > 0$, this has dimension $n-|S|$, and these subspaces are the affine
 linear subspaces supporting the faces of the simplex $\Delta_M$.

 Choose $M$ generic so that for all $S$ the subspace $H_S$ meets $Z$ transversally. 
 For each subset $S$, set $Z_S := Z \cap H_S$. This is a smooth hypersurface in $H_S$ and  
 therefore has dimension $n-|S|-1$. The boundary stratum $Y_S$ of $Y = Z \cap \Delta_M$
 lying in the relative interior of the face supported by $H_S$ is an open subset of $Z_S$.

 For a nonzero vector $u \in \R^n$, the directional derivative $D_u\varphi$ is
 \[
  \sum_{\alpha \in \calA} (u \cdot \alpha)c_{\alpha}e^{z \cdot \alpha}
 \]
 which is an exponential sum having the same exponents as $\varphi$. Let $L_u$ be the
 linear function on $\R^n$ defined by $z \mapsto u \cdot z$.

 The critical points of the function $L_u$ restricted to $Z$ are the zeroes of the
 system
 \[
  \varphi(z) = 0 \quad \text{ and } \quad 
  D_v \varphi(z) = 0 \quad \text{ for } v \in u^{\perp}\,.
 \]

 When $u$ is general and we choose a basis for $u^\perp$, this becomes a system of
 $n$ exponential sums in $n$ variables having the same support as the original polynomial.
 Therefore, the whole system has $(d,\ell)$-dense support,
 $\calA=\psi(d\simplex^\ell\cap\Z^\ell)\cup \{e_1,\dotsc,e_n\}$.
 By Theorem~\ref{Th:Structured_bound}, the number of solutions is at most
 \[
   \frac{e^2{+}3}{4}\, 2^{\binom{\ell}{2}}\, n^\ell \cdot d^\ell\,.
 \] 

 We use this to estimate the number of critical points of the function $L_u$ restricted
 to $Z_S$. 
 The restriction  $\varphi_S$ of $\varphi$ to $H_S$ defines $Z_S$ as a hypersurface in $H_S$.
 We determine this restriction. 
 Suppose first that $0\notin S$.
 If $i\in S$, then we may use  the equation $z_i = -M_i$ to eliminate the variable $z_i$ and
 the exponential $e^{z_i}$  from $\varphi$. 
 The effect of these substitutions for $i\in S$ on the exponents is the projection $\pi$
 sending $e_i\mapsto 0$ for $i\in S$.
 Then $\pi \circ \psi$ is still affine and so 
 $\varphi_S$ is still a $(d,\ell)$-dense fewnomial but with $n$ replaced by $n{-}|S|$,
 and thus the number of critical points of $L_u|_{H_S}$ on $Z_S$ is bounded by 
\[
   \frac{e^2{+}3}{4}\, 2^{\binom{\ell}{2}}\, (n-|S|)^\ell \cdot d^\ell\,.
\]

 If $0 \in S$, then we could use fewnomial theory to estimate the number of critical
 points of $L_u|_{H_S}$ on $Z_S$, but will not need that estimate.

 Let $u$ be a general vector in $\R^n$ such that $L_u$ is a Morse function for the 
 stratified space $Y$. By Proposition $2$ in~\cite{BS_Betti}, the sum of the Betti numbers
 of $Y$ is bounded by the number of critical points $p$ of $L_u$ for which $L_u$ achieves
 its minimum on the normal slice $N(p)$ at $p$. Since the strata $Y_S$ of $Y$ are open
 subsets of the manifolds $Z_S$, this number is bounded above by the number of such
 critical points of $L_u$ on the manifolds $Z_S$. Just as in~\cite{BS_Betti}, we can alter
 $u$ so that no critical point in any $Z_S$ with $0 \in S$ contributes. Therefore, 
 the sum of the Betti numbers of $Y$ is bounded above by
 \[
  \frac{e^2{+}3}{4}\, 2^{\binom{\ell}{2}} \cdot d^\ell \cdot
  \sum_{S \subset \{1,\dots,n\}} (n-|S|)^\ell 
   \ =\ 
   \frac{e^2{+}3}{4}\, 2^{\binom{\ell}{2}} \cdot d^\ell \cdot
  \sum_{i=0}^{n} {\binom{n}{i}} (n-i)^\ell\,.\vspace{-30pt}
 \]
\end{proof}\vspace{15pt}

%%%%%%%%%%%%%%%%%%%%%%%%%%%%%%%%%%%%%%%%%%%%%%%%%%%%%%%%%%%%%%%%%%%%%%%%%
\providecommand{\bysame}{\leavevmode\hbox to3em{\hrulefill}\thinspace}
\providecommand{\MR}{\relax\ifhmode\unskip\space\fi MR }
\providecommand{\MRhref}[2]{%
  \href{http://www.ams.org/mathscinet-getitem?mr=#1}{#2}
}
\providecommand{\href}[2]{#2}

%%%%%%%%%%%%%%%%%%%%%%%%%%%%%%%%%%%%%%%%%%%%%%%%%%%%%%%%%%%%%%%%%%%%%%%%%
\end{document}